\documentclass[a4paper]{article}

\usepackage{graphicx}

\usepackage{amssymb}
\usepackage{amsmath}
\usepackage{amsthm}
\usepackage[linesnumbered,ruled,vlined]{algorithm2e}
\usepackage{tikz}
\usetikzlibrary{decorations.pathreplacing}

\usepackage{hyperref}

\newtheoremstyle{it_dotless} 
                        {0.5em}   
                        {0.5em}   
                        {\itshape}  
                        {}          
                        {\bfseries} 
                        {:}         
                        {\newline}  
                        {}			

\newtheoremstyle{dotless} 
                        {0.5em}   
                        {0.5em}   
                        {}  		
                        {}          
                        {\bfseries} 
                        {:}         
                        {\newline}  
                        {}			
                        
\theoremstyle{it_dotless}
\newtheorem{theorem}{Theorem}
\newtheorem{lemma}[theorem]{Lemma}
\newtheorem{corollary}[theorem]{Corollary}

\theoremstyle{dotless}
\newtheorem{remark}[theorem]{Remark}
\newtheorem{example}[theorem]{Example}

\newtheorem{definition}[theorem]{Definition}

\newcommand{\bitem}{\begin{itemize}}
\newcommand{\eitem}{\end{itemize}}
\newcommand{\mc}[1]{\mathcal{#1}}

\newcommand{\N}{\mathbb{N}}
\newcommand{\R}{\mathbb{R}}

\newcommand{\M}{\mathcal{M}}

\newcommand{\bpm}{\begin{pmatrix}}
\newcommand{\epm}{\end{pmatrix}}
\newcommand{\bsm}{\left(\begin{smallmatrix}}
\newcommand{\esm}{\end{smallmatrix}\right)}
\newcommand{\T}{\top}

\newcommand{\la}{\langle}
\newcommand{\ra}{\rangle}

\newcommand{\dotbigcup}{\mathop{\dot{\bigcup}}}

\DeclareMathOperator{\intr}{int}

\DeclareMathOperator{\bd}{bd}

\DeclareMathOperator*{\argmax}{arg max}
\DeclareMathOperator{\supp}{supp}

\DeclareMathOperator{\conv}{conv}

\DeclareMathOperator*{\var2}{\mathcal{V}_{\mathbb{R}}}

\DeclareMathOperator*{\td}{td}

\DeclareMathOperator*{\FPC}{FPC}
\DeclareMathOperator{\rnd}{rnd}

\begin{document}

\title{Symmetry-free SDP Relaxations for \\ Affine Subspace Clustering
\thanks{Authors gratefully acknowledge support by the German Research Foundation (DFG), grant GRK 1653.
}
}

\author{Francesco Silvestri\footnote{Institut f\"{u}r Informatik, Heidelberg University, INF205, 69120 Heidelberg, Germany}$^{\; ,}$\footnote{IWR, Heidelberg University, INF205, 69120 Heidelberg, Germany} \and Gerhard~Reinelt \footnotemark[2] \and Christoph Schn\"{o}rr \footnote{Institut f\"{u}r Angewandte Mathematik, Heidelberg University, INF205, 69120 Heidelberg, Germany}}

\maketitle

\begin{abstract}
We consider clustering problems where the goal is to determine an optimal partition of a given point set in Euclidean space in terms of a collection of affine subspaces. 
While there is vast literature on heuristics for this kind of problem, such approaches are known to be susceptible to poor initializations and getting trapped in bad local optima. 
We alleviate these issues by introducing a semidefinite relaxation based on Lasserre's method of moments. While a similiar approach is known for classical Euclidean clustering problems, a generalization to our more general subspace scenario is not straighforward, due to the high symmetry of the objective function that weakens any convex relaxation. We therefore introduce a new mechanism for symmetry breaking based on covering the feasible region with polytopes. Additionally, we introduce and analyze a deterministic rounding heuristic.

\end{abstract}

\section{Introduction}

\subsection{Background}

Given points $b_{i} \in \R^{d},\, i \in [n]:=\{1,2,\dotsc,n\}$ and $k\in \N$, the classical Euclidean clustering problem asks to jointly minimize the objective 
\begin{equation} \label{eq:intro-clustering-1}
\min_{u, x}  \sum_{i\in [n]} \sum_{j \in [k]} u_{ij} \| x_j - b_i \|^2
\end{equation}
with respect to \emph{centroids} $x_{j}\in \R^{d},\, j \in [k]$, and \emph{assignment variables} $u_{ij} \in \{0,1\}$ such that each row of the \emph{assignment matrix} $U=(u_{ij})_{i,j}\in \{0,1\}^{n\times k}$ only contains a single one. 

While optimizing over \emph{both} sets of variables \emph{jointly} is known to be NP-hard due to its nonlinear, combinatorial structure, fixing one set of variables immediately leads to trivial subproblems. For this reason, many popular heuristics ($k$-means, mean-shift, etc. -- cf.~\cite{Xu2005}) focus on alternatingly fixing one set of variables while optimizing the remaining ones. Although these heuristics are generally easy to implement and fast, they strongly depend on proper initializations and may easily get stuck in local optima without any approximation guarantees. 

This shortcoming can be avoided by turning to combinatorial optimization techniques (e.g.~\cite{Merle-et-al-00}), but they do not scale up to large data sets. 

Alternatively, a semidefinite convex relaxation \cite{KMeans-SDP-07} has been suggested. This approach is remarkable in that it avoids the inherent problems with symmetry of \eqref{eq:intro-clustering-1} for convex relaxations:\newline
Given any solution $\{ x_{j} \}_{j\in[k]}$, $U=[U_1,\ldots,U_k]$ of \eqref{eq:intro-clustering-1}, as well as a permutation $\pi$ of $[k]$, we get another solution $\{ x_{\pi(j)} \}_{j\in[k]}$, $[U_{\pi(1)},\ldots,U_{\pi(k)}]$ with equal objective value. For this reason, convex relaxations tend to average 
over optimal solution 
through convex combinations, making it nearly impossible to recover information about the individual clusters.

The approach of \cite{KMeans-SDP-07} avoids this issue by reducing the problem to a linear program over projection matrices constructed from $U$, which can be effectively approximated by SDPs. This reduction however essentially depends on the \textit{specific simple closed form solution} of optimal centroids $x_j$, when the assignment variables $U$ are fixed.

In this paper, we focus on convex relaxations of the significantly more involved problem
\begin{equation} \label{eq:intro-clustering-2}
\min_{u, x}  \sum_{i\in [n]} \sum_{j \in [k]} u_{ij} \| A_{i}x_j-b_i\|^2_2
\end{equation}
with \emph{given} data $(A_{i},b_i)\in \R^{l\times d}\times \R^l,\, i \in [n]$, \emph{unknown} parameters $x_{j} \in \R^{d},\, j \in [k]$, and an \emph{unknown} assignment matrix $U=(u_{ij})_{i,j} \in \{0,1\}^{n\times k}$. In comparison with \eqref{eq:intro-clustering-1}, this approach extends the representation of data by \emph{points} to \emph{affine subspaces}, which is significant for many applications.

Not surprisingly, the approach of \cite{KMeans-SDP-07} cannot be adapted to this advanced setting. While it is possible, of course, to give the closed form of an optimal $x_j$ for fixed $U$, namely
\begin{equation}
x_j(U) = \Big(\sum_{i\in [n]} u_{ij}A_i^{\T}A_i\Big)^\dagger \Big(\sum_{i\in [n]} u_{ij}A_i^{\T} b_i\Big),\label{closedFormula}
\end{equation}
this closed form involves pseudo inverses $(\ldots)^\dagger$ of linear functions of $U$ and hence does not yield any exploitable structures for the reduced problem. In fact, due to the nonlinearity of $(\ldots)^\dagger$, it is not even clear how to express $x_j(U)$ as a rational function in $U$ without explicitly computing each $x_j(U)$ for every possible choice of $U$ beforehand.

Unions of subspaces as signal models have been advocated and studied in the research field of compressive sensing during the last years \cite{Carin2011}. In order to prove recovery guarantees by convex programming, sparsity assumptions about the representation are essential. Regarding the subspace clustering problem, such assumptions require the subspace dimensions to be low relative to the dimension of the embedding space. In this paper, we do not rely on any such assumption. For example, even the simplest case of clustering one-dimensional subspaces in $\R^2$ violates the ``independent subspaces'' assumption of \cite[Section 4]{Elhamifar2012}.

\subsection{Contribution}

Our main contribution is a hierarchy of convex relaxations for problem \eqref{eq:intro-clustering-2} based on Lasserre's method of moments that avoids the degeneracy of solutions induced by symmetry. 

Our approach is based on the assumption that we can cover the feasible region with polytopes in such a way that each optimal center $\{x_j\}_{j\in[k]}$ is covered by the interior of exactly one polytope. 
Under this assumption, we are able to reduce \eqref{eq:intro-clustering-2} to a highly structured optimization problem over a constrained simplex. Using this new structure, symmetric solutions can be relaxed away, and Lassere's method of moments can be used to give a hierarchy of convex relaxations. 


\subsection{Organization of the Paper}

We summarize Lasserre's method of moments in section \ref{sec:notation} and introduce the notation.
In section \ref{sec:k-clustering}, we reduce \eqref{eq:intro-clustering-2} to a highly structured optimization problem~$(R1)$ over constrained simplices in order to derive the symmetry-free formulation~$(R2)$. 

Section \ref{sec:relaxation} is mostly devoted to the application of Lasserre's method of moments to problem $(R2)$, which yields the hierarchy $(R2)[t]$. After pointing out ways to simplify $(R2)[1]$, we also suggest a relaxed hierarchy $(R3)[t]$ that can be computed much faster.

To complete the algorithmic procedure, we also give a deterministic rounding heuristic in section \ref{sec:rounding}. 

Finally, based on our novel approach, we sketch in section \ref{sec:extention} several ways to extend \eqref{eq:intro-clustering-2} to more general settings. Some experiments are reported in section \ref{sec:experiments} to illustrate the mechanism for symmetry breaking, that is essential for effective SDP relaxation.


\section{Preliminaries}\label{sec:notation}

This section gives a basic description of  \textbf{Lasserre's Method of Moments} and is based mostly on the book \cite{lasserre15}, with some minor changes of notation.

\subsection{Linear Algebra}

In the following, we list some cones with their corresponding partial order as
\begin{itemize}
\item $(\R^n_+,\leq)$ vectors in $\R^n$ with nonnegative entries,
\item $(\mc{S}^n_+,\preceq)$ symmetric positive semidefinite $n\times n$ matrices,
\item $(\mc{D}^n,\preccurlyeq)$ double nonnegative matrices given as $\mc{D}^n=\mc{S}_+^n\cap \R^{n\times n}_+  $.
\end{itemize}

Special matrices include the $n\times n$ identity $I_n$ and the $n\times n$ all ones matrix $J_n=ee^\T$ where $e$ denotes the vector of all ones of appropriate dimension.

\subsection{Polynomials}

Given a vector $\alpha\in \N^d$ and $x\in\R^d$ define the monomial $x^\alpha = \prod_{i\in [d]}x_i^{\alpha_i}$ and its total degree as $\deg(x^\alpha)=\langle e,\alpha\rangle$. Let $\R[x]$ denote the set of multivariate polynomials in $x$ where we set $\deg(p):=\max\{\deg(x^\alpha) \colon p_\alpha\neq 0\}$ for any element $p\in \R[x]$.

Furthermore, the vector space of polynomials of degree at most $t$ is given as 
\begin{equation}
\R_t[x]=\{p\in \R[x]\colon \deg(p)\leq t\}
\end{equation}
where 
\begin{equation}
z(t):=\dim(\R_t[x])=\binom{d+t}{d}.
\end{equation}
By defining
\begin{equation}
\N^d_t:=\{\alpha\in \N^d \colon \langle \alpha,e\rangle \leq t \}
\end{equation}
we see that each polynomial $p\in \R_t[x]$ can be written as  $p(x)=\sum_{\alpha \in \N^d_t}p_\alpha x^\alpha$ and we may identify $\R_t[x]$ with $\R^{z(t)}$ by treating $p$ as the vector of its coefficients. In this context we will also write $p\in\R^{z(t)}$ and encode the canonical monomial base $(x^\alpha)_{\alpha\in \N^d_t}$ as $v_d(x)$ such that $p(x)=\la p, v_d(x)\ra$.


Given $y=(y_\alpha)_{\alpha\in \N^d}$, we can use this identification to define the \textbf{Riesz functional} $L_y:\R[x]\rightarrow\R$ as $p\mapsto L_y(p)=\langle p,y\rangle$. 

\subsection{Moment matrices}

For $t\in\N$ and $y\in \R^{\N^d_{2t}}$, the matrix $M_t(y)$ of size $z(t)$ is defined by
\begin{equation}
(M_t(y))_{\alpha,\beta}:=L_y(x^\alpha \cdot x^\beta)=y_{\alpha+\beta}
\end{equation} 
and is called the \textbf{moment matrix} of order $t$ of $y$.

\noindent More generally, let $f$ be a multivariate polynomial and define $\td(f):=\lceil \frac{\deg(f)}{2}\rceil$. For $t\geq \td(f)$,  the matrix $M_{t}(f,y)$ of size $z(t-\td(f))$ is defined by
\begin{equation}
(M_{t}(f,y))_{\alpha,\beta}:=L_y(x^\alpha\cdot x^\beta \cdot f)
\end{equation} 
and is called the \textbf{localizing moment matrix} of order $t$ of $y$ with respect to $f$.
Note that each entry of $M_t(f,y)$ is a linear expression in $y$ and that we recover $M_t(y)=M_t(1,y)$ as a special case.

\subsection{Measures and moments}

Let $\mc{N}(K)\subseteq \R_t[x]$ be the convex cone of polynomials that are nonnegative on~$K$ and denote the dual cone by
\begin{equation}
\mc{N}^*(K)=\left\{ y\in\R^{\N^d} \,\middle|\, L_y(f)\geq 0,\, \forall f\in \mc{N}(K)\right\}.
\end{equation}

For a set $K\subseteq \R^d$, denote by $\M_+(K)$ the space of finite (nonnegative) Borel measures supported on~$K$ and by $\mc{P}(K)$ the subset of probability measures on $K$. We can recover the cone of the corresponding moments
\begin{equation}
\left\{ y\in\R^{\N^d} \,\middle|\, \exists \mu\in\mc{M}_+(K)\colon y_\alpha = \int_K x^\alpha d\mu \quad \forall \alpha\in \N^d \right\}\subseteq \mc{N}^*(K)\label{eq:cone_of_moments}
\end{equation}
where equality holds if $K$ is compact \cite[Lemma 4.7]{lasserre15}.

\subsection{Reformulation of Optimization Problems}\label{sec:Reformulation_of_Optimization}

\noindent Let $K\subseteq \R^d$ be a compact set and $f(x)= \sum_{\alpha \in \N^d_t}f_\alpha x^\alpha$ be a real-valued multivariate polynomial, then 
\begin{equation}
\inf_{x\in K} f(x)=\inf_{\mu\in \mc{P}(K)}\int_K f d\mu \label{eq:inf_problem}
\end{equation}
can be reduced to a convex linear programming problem. Indeed, we have that 
\begin{equation}
\int_K f d\mu = \int_K \sum_{\alpha \in \N^d_t}f_\alpha x^\alpha d\mu =  \sum_{\alpha \in \N^d_t}f_\alpha \int_K x^\alpha d\mu = L_y(f)
\end{equation} where $y_\alpha = \int_K x^\alpha d\mu$ is the moment of order $\alpha$.

Consequently, if $f$ is polynomial, then
\begin{equation}
\inf L_y(f) \quad s.t.\quad  y_0=1,\; y\in \mc{N}^*(K) \label{eq:inf_convex}
\end{equation}
is a relaxation of problem \eqref{eq:inf_problem} with the benefit of being a reformulation whenever equality holds in \eqref{eq:cone_of_moments}.

Note that the constraint $y_0=1$ enforces that $y$ represents a measure in $\mc{P}(K)\subsetneq\mc{M}_+(K)$, provided $y\in \mc{N}^*(K)$.

\noindent Although problem \eqref{eq:inf_convex} is a convex linear programming problem, the characterization of $y\in \mc{N}^*(K)$ (known as $K$-moment problem in the literature) may be notoriously hard for general $K$.

However, for compact semi-algebraic $K$ given as 

\begin{equation}
K=\{x\in \R^d \colon g_i(x)\geq 0\quad  \forall i\in [k]\},
\end{equation}

\noindent for some polynomials $g_i \in \R[x]$, an explicit characterization of $\mc{N}^*(K)$ is available.
Since $K$ is assumed to be compact, we will assume without loss of generality that 
\begin{equation}
g_1(x)=R^2-\|x\|^2\geq 0,
\end{equation}
where $R$ is a sufficiently large positive constant (in fact, we would only need any function $u$ in the quadratic module generated by $\{g_i\}_{i\in [k]}$ to have a compact superlevel set $\{x\in \R^d \colon u(x)\leq 0\}$ for the following). This representation allows the application of a theorem on positivity by Putinar \cite[Theorem 2.15]{lasserre15}, which leads to
\begin{align}
 \mc{N}^*(K) &=\{y\in\R^{\N^d}\colon M_t(y)\succeq 0,\; M_t(g_i, y)\succeq 0\quad\forall i\in [k],\; \forall t\in \N\}\\
 &=: \mc{N}_\succeq^* (g_1,\ldots,g_k).\label{eq:putinar}
\end{align}
In particular, problem \eqref{eq:inf_convex} is equivalent to
\begin{equation}
\inf_{y\in \R^{\N^d}} L_y(f) \quad s.t. \quad y_0=1,\; y\in \mc{N}^*_\succeq (g_1,\ldots,g_k).
\end{equation}

To summarize, if $f$ is polynomial and $K$ a compact semi-algebraic set, then problem \eqref{eq:inf_problem} is equivalent to a convex linear programming problem with an infinite number of linear constraints on an infinite number of decision variables. 

\subsection{Semidefinite Relaxations}\label{sec:sdp_relax}

\noindent Now, for $t\geq \td(f)$, consider the finite-dimensional truncations 
\begin{equation}
\rho_t=\inf_{y\in \R^{\N^d_{2t}}} L_y(f) \quad s.t.\quad y_0=1,\, y\in \mc{N}_t^*(g_1,\ldots,g_k)\label{eq:LMM}
\end{equation}
of problem \eqref{eq:inf_convex} where
\begin{equation}
\mc{N}^*_t(g_1,\ldots,g_k) :=
\left\{y\in \R^{\N^d_{2t}} \,\middle|\,
\begin{array}{rl}
M_t(y)& \succeq 0,\\
M_{t}(g_i, y) &\succeq 0\quad \forall i\in[k]: t\geq \td(g_i)
\end{array}
\right\}.
\end{equation}

\noindent By construction, $\{\mc{N}^*_t\}_{t\in\N}$ generates a hierarchy of relaxations of Problem \eqref{eq:inf_convex}, where each $\{\mc{N}^*_t\}_{t\in\N}$, is concerned with moment and localizing matrices of fixed size $t$. The lowerbounds $\rho_t$ monotonically converge toward the optimal value of \eqref{eq:inf_problem} \cite[Theorem 6.2]{lasserre15} and finite convergence may take place, which can be efficiently checked \cite[Theorem 6.6]{lasserre15}. 

Furthermore, in the best case of finite convergence, \eqref{eq:LMM} will yield the global optimal value and a convex combination of global optimal solutions as minimizer, which can be efficiently decomposed into optimal solutions \cite[Sct. 6.1.2~]{lasserre15}.

In the noncompact case, the $\rho_t$ are still monotonically increasing lower bounds of \eqref{eq:inf_problem}, but convergence to the optimum is not guaranteed.

\begin{remark}
In the literature, this construction is known as \textbf{Lasserre's Method of Moments} (LMM) where it is assumed that $t \geq \max_{i}  \td(g_i)$ in addition to $t\geq \td(f)$ in order to start with a complete description of all the constraints used in the problem. Our slightly different definition is more flexible by enabling us to start with an incomplete set of constraints of low degree while still fitting into the overall hierarchy. 

It should be noted that using a value of $t$ that truncates most of the 'relevant' inequalities for the problem is not likely to yield a useful lower bound. 
\end{remark}

For convenience, we will also introduce a shortcut notation for polynomial equations $h(x)=0$ (imposed by having both $h(x)\geq 0$ and $-h(x)\geq 0$) by setting
\begin{equation}
\mc{N}^*_t(\{h_j\},\{g_i\}):=
\left\{y\in \R^{\N^d_{2t}} \,\middle|\,
\begin{array}{rl}
 M_t(y)&\succeq 0 \\
 M_{t}(h_i, y)& = 0\, \quad \forall j: t\geq \td(h_j)\\
 M_{t}(g_i, y) & \succeq 0\,\quad \forall i: t\geq \td(g_i)
\end{array}
\right\}.
\end{equation}


\section{Dealing with the Symmetry of the $k$-Clustering Problem}\label{sec:k-clustering}

This section starts by outlining the problem of $k$-clustering and the associated difficulties in solving it, in section \ref{sub:formulation}. 
In section \ref{sub:reformulation}, we preprocess the $k$-clustering problem by reducing it to a quadratic optimization problem over a simplex with an additional partition structure.
We then use this description in section \ref{sub:break symmetry} as a basis to relax the partition constraints in a way that removes symmetric solutions.

\subsection{Problem Formulation} \label{sub:formulation}

We study \eqref{eq:intro-clustering-2} in the form
\begin{subequations}\label{eq:ProblemFormulation}
\begin{align}
\min_{u, x}  & \sum_{i\in[n]} \sum_{j\in [k]} u_{ij} \| A_i x_j - b_i \|^{2}_2 \\
\text{s.t.} \enspace & Ue=e,\quad U\in \{0,1\}^{n\times k},
\end{align}
\end{subequations}
where $\{A_i\}\subseteq \mathbb{R}^{l\times d}$ and $\{b_i\}\subseteq \mathbb{R}^{l}$. 
Since 
\begin{equation}
\| A_i x_j - b_i \|^{2}= x_j^\T (A_i^\T A_i) x_j - 2(b_i^\T A_i) x_j + \|b_i\|^2_2
\end{equation}
and 
\begin{equation}
u_{ij}\in \{0,1\}\quad \Leftrightarrow  \quad u_{ij}(1-u_{ij})=0
\end{equation}
for all $i\in [n], j\in[k]$, we see that \eqref{eq:ProblemFormulation} asks us to optimize a polynomial over a real variety.

By assuming that all sensible solutions $\{x_j\}_{j\in[k]}$ are contained in a compact set $K$, we could apply LMM from the preceding section in order to approximate the solution of this problem.

However, due to increasing size we cannot expect to compute the level of convergence in LMM. In particular, it is hard to extract feasible solutions from lower levels of LMM, and the symmetric structure of the partition matrix $U$ makes this even harder. 

For example, consider any permutation $\pi\in \mathfrak{S}_k$ and an optimal solution $(U^*,X^*)$ to \eqref{eq:ProblemFormulation}. Then one can check that the values $(U^\pi,X^\pi)$ where $u^\pi_{ij}:=u^*_{i\pi(j)}$ and $x^\pi_j:=x^*_{\pi(j)}$ are an optimal solution for \eqref{eq:ProblemFormulation}, which corresponds to relabeling the clusters. Furthermore, $(U',X')$ given by
\begin{equation}
(U',X')= \frac{1}{k!}\sum_{\pi\in\mathfrak{S}_k} (U^\pi,X^\pi)
\end{equation}
will be a valid solution for each step of LMM. Since $u_{ij}'=\frac{1}{k}$ for all $i\in[n], j\in[k]$ and $x'_i=x'_j$ for all $i,j\in[k]$, there is no way to recover an optimal assignment. Since we need the assignment as well, we will reformulate the problem in the next section to avoid this symmetry.

\subsection{Parametrization with Constrained Simplices} \label{sub:reformulation}

Throughout this paper, we will assume that the feasible region can be covered by a finite set of polytopes, which is a reasonable assumption since the feasible region of most practical problems are bounded \cite{Xu2005}.
Section \ref{sec:extention} will comment on more elaborate ways to parametrize the feasible region using simplices. We therefore start with the following central assumption.\\

\noindent \textbf{Triangulation Assumption:}

\noindent \emph{The optimal solution $\{x_j\}_{j\in [k]}$ to \eqref{eq:ProblemFormulation} is contained in a union of simplices, e.g.
\begin{equation}
\{x_j\}_{j\in [k]}\subseteq \mc{P}=\bigcup_{s\in [q]} P_s\label{eq:triangulation}
\end{equation}
is a valid constraint for \eqref{eq:ProblemFormulation} where the $\{P_s\}_{s\in[q]}$ are $d$-dimensional simplices with disjoint interior. Furthermore, $\mc{P}$ can be constructed from $\{(A_i, b_i)\}_{i\in[n]}$.}\\
%

\noindent To exploit this, let $V_s$ be the matrix whose rows denote the vertices of $P_s$ such that $\conv(V_s)=P_s$. Let 
\begin{equation}
m:=\sum_{s\in [q]} |V_s|= q(d+1)
\end{equation}
such that for $\lambda^\T:=(\lambda_{v(1)}^\T,\ldots,\lambda_{v(q)}^\T)\in \Delta^{m}$ and $V=(V_1,\ldots,V_q)$ we have
\begin{equation}
x=V\lambda=\sum_{s\in [q]} V_s\lambda_{v(s)}.
\end{equation}

\begin{remark}\label{rem:redundantColumns}
Note that if the simplices $P_s$ have common vertices, $V$ will have multiple identical columns across different $V_s$. This is done on purpose, as the removal of redundant copies will be treated in section \ref{sec:extention}.
\end{remark}

\noindent Then we can express $\mc{P}$ as the image of $\Delta^m$ constrained by
\begin{equation}\label{eq:Parametrization}
\mc{P} =\{x\in\R^d \,|\, \exists \lambda\in \Delta^{m}\colon  x=V\lambda,\quad  \lambda_{v(r)} \lambda_{v(s)}^\T=0 \quad\forall\, r,s\in[q], r\neq s\}.
\end{equation}

\noindent The nonlinear orthogonality constraint 
\begin{equation}\label{eq:OrthogonalityConstraints}
\lambda_{v(r)}\lambda_{v(s)}^\T=0 \quad\forall\, r,s\in[q],r\neq s
\end{equation}
ensures that exactly one $\lambda_{v(s)}$ is nonzero, which implies $x=V\lambda=V_s\lambda_{v(s)}\in P_s$. Since $\lambda\geq 0$, we can see that \eqref{eq:OrthogonalityConstraints} is equivalent to the sum 
\begin{equation}
\lambda^\T \Omega \lambda=0
\end{equation}
where 
\begin{equation}
\Omega=(J_q-I_q)\otimes J_{d+1}
\in\{0,1\}^{m\times m}
\end{equation}
is given by a Kronecker product and zero on a block diagonal. This suggests to set  
\begin{equation} \label{eq:def-Delta-m-Omega}
\Delta^m_\Omega := \{ \lambda\in \Delta^m \colon \lambda^\T \Omega \lambda =0\}
\end{equation}
so that in particular, we can write $\mc{P}=V\Delta^m_\Omega$ as a shorthand for \eqref{eq:Parametrization}.

\begin{remark}
\noindent Note that unless $q=1$, the representation of $\mc{P}$ in \eqref{eq:Parametrization} is in general not unique since our assumption does not exclude the case that the intersection of boundaries $\bd(P_r)\cap \bd(P_s)$ is nonempty. However, by Caratheodory's theorem \cite[Thm.~2.29]{Rockafellar2009} we get a unique representation for all interior points $x\in \dotbigcup_{s\in [q]} \intr(P_s)$.
%
\end{remark}

%
Using the parametrization \eqref{eq:Parametrization} in \eqref{eq:ProblemFormulation} and using $1=\langle \lambda, e\rangle$ to homogenize, we can use $x_j=V\lambda^j$ to rewrite
\begin{equation}
\| A_i x - b_i \|^2_2=\la \lambda^j, W_i \lambda^j\ra,
\end{equation}
where
\begin{equation}
W_i:=V^{\T} A_i^{\T} A_i V - (e b_i^{\T} A_i V + V^\T A_i^\T b_i e^\T) + \|b_i\|^2_2 \cdot J_m
\end{equation}
for all $i\in [n]$. This leads to the reformulation

\begin{subequations}\label{eq:FirstReformulation}
\begin{align}
\text{(R1)}\quad  \min_{u, \lambda}  & \sum_{i\in[n]} \sum_{j\in [k]} u_{ij}  \la \lambda^j, W_i \lambda^j\ra	 \\
\text{s.t.} \enspace
& Ue=e,\quad\quad\quad U\in \{0,1\}^{n\times k}, \\
& \lambda^j\in\Delta^{m}_\Omega\quad \forall j\in[k].
\end{align}
\end{subequations}

\subsection{Removing Symmetry with Separating Triangluations}\label{sub:break symmetry}

The goal of this section is to eliminate the variable $U$ in (R1). To this end, recall that the purpose of $U$ is to model that each term $W_i$ is only evaluated at a single point in $\{\lambda^j\}_{j\in[k]}$. In particular, (R1) models the problem
\begin{equation}
\min_{\lambda}   \sum_{i\in[n]} \la \lambda_i, W_i \lambda_i\ra 
\quad \text{s.t.} \quad  \lambda_i \in \{\lambda^j\}_{j\in[k]}\subseteq \Delta^{m}_\Omega\quad \forall i\in[n], \label{eq:membership_constraint}
\end{equation}
where the membership constraint $\lambda_i \in \{\lambda^j\}_{j\in[k]}$ is modeled by $U$. 

While it is important to model this membership constraint in a more tractable formulation for optimization, using the $U$ variable introduces the inherent problematic symmetries mentioned in Section \ref{sub:formulation} in the first place by turning the unordered set $\{ \lambda^j\}_{j\in[k]}$ into an arbitrarily ordered list. 

As a main idea of this paper, we propose another, symmetry free formulation for this membership based on the following property.

\begin{definition}[Separating Triangulation]
Let $\mc{P}=\bigcup_{s\in [q]} P_s$ be a triangulation satisfying \eqref{eq:triangulation}. A set of points $\{x_j\}_{j\in [k]}\subseteq \mc{P}$ is called \textbf{separated} by $\mc{P}$ if 
\begin{equation}
|\{x_j\}_{j\in [k]} \cap  P_s| \leq 1 \quad \forall s\in [q].\label{eq:separation_criterion}
\end{equation}
$\mc{P}$ is called \textbf{separating} (for \eqref{eq:ProblemFormulation}) if an optimal set of centroids $\{x_j\}_{j\in [k]}$ for \eqref{eq:ProblemFormulation} is separated by $\mc{P}$.
\end{definition}

The central advantage of a separating triangulation is that the representation of the optimal $\{x_j\}_{j\in [k]}$ in (R1) becomes orthogonal. 

\begin{lemma}\label{lem:orthogonal}
Let $\mc{P}=\bigcup_{s\in [q]} P_s$ be a triangulation satisfying \eqref{eq:triangulation}. For a set of points $\{x_j\}_{j\in [k]}\subseteq \mc{P}$, let $x_j=V\lambda^j$ with $\lambda^j\in \Delta_\Omega^m$ for all $j\in [k]$ be their representation in (R1). Then $\{x_j\}_{j\in [k]}$ is separated if and only if whenever $j, j'\in [k]$ and  $j\neq j'$, 
\begin{equation}
\lambda^j_{v(s)} \left( \lambda^{j'}_{v(s)}\right)^\T =0 \quad \forall s\in [q].\label{eq:orthogonal_representation}
\end{equation}
In particular, their representations are coordinatewise orthogonal. Furthermore, if $\mc{P}$ is separating, then \eqref{eq:orthogonal_representation} holds for an optimal solution of (R1).
\end{lemma}

This simple observation implies that we can encode the membership constraint in \eqref{eq:membership_constraint} with linear inequalities and quadratic constraints.

\begin{theorem}\label{thm:membership}
Let $\mc{P}$ in \eqref{eq:triangulation} be separating and let 
\begin{equation}
\lambda_* := \sum_{j\in [k]}\lambda^j \label{eq:*definition}
\end{equation}
for the optimal solution $\{\lambda^j\}_{j\in [k]}$ in (R1). Then for $\lambda\in \Delta^m_\Omega$ we have the equivalence
\begin{equation}
\lambda \in \{\lambda^j\}_{j\in[k]}\quad \Leftrightarrow \quad \lambda\leq \lambda_*.\label{eq:containment}
\end{equation}
\end{theorem}
\begin{proof}
Implication $" \Rightarrow "$ is straightforward. Consider the reverse direction. Denoting by $\supp(\lambda)$ the support of $\lambda$, it follows from $\lambda^\T \Omega \lambda =0$, $0\leq \lambda \leq \lambda_*$  and Lemma \ref{lem:orthogonal} that we have 
\begin{equation}
\supp(\lambda) \subseteq v(s_{j'})\subseteq \supp(\lambda_*)=\dot{\bigcup_{j\in [k]} } \supp(\lambda^j)= \dot{\bigcup_{j\in [k]} } v(s_j)
\end{equation}
for some $j'\in [k]$, where $\supp(\lambda^j)=v(s_j)$ for all $j\in[k]$. Therefore,
\begin{equation}
0\leq \lambda \leq \lambda^{j'} \quad\text{ and }\quad \la \lambda,e\ra = 1 = \la \lambda^{j'},e\ra,
\end{equation}
which is only possible if $\lambda = \lambda^{j'}$.
\end{proof}

Theorem \ref{thm:membership} reduces the membership in $\{\lambda^j\}_{j\in k}$ to a more tractable relation involving $\lambda_*$. However, since $\lambda_*$ encodes the variables that we try to optimize, we still need to give a proper characterization of those $\lambda_*$ corresponding to separated solutions $\{\lambda^j\}_{j\in [k]}$. 

\begin{lemma}
Let 
\begin{equation}
\mc{L}:= \left\{ \{\lambda^j\}_{j\in[k]}\subseteq \Delta_\Omega^m \,|\, \eqref{eq:orthogonal_representation}\text{  holds}\right\}
\end{equation}
and
\begin{equation}
\mc{L'}:= \{\lambda\in k\cdot \Delta^m \,|\, \la \lambda_{v(s)},e\ra \lambda_{v(s)} = \lambda_{v(s)} \quad \forall s\in [q]\}.\label{eq:L'}
\end{equation}
Then $\mc{L}$ and $\mc{L'}$ are in one-to-one correspondence with a bijection $\phi: \mc{L} \rightarrow \mc{L'}$ given by 
\begin{equation}
\phi(\{\lambda^j\}_{j\in[k]}):=\sum_{j\in [k]} \lambda^j.
\end{equation}
\end{lemma}
\begin{proof}
It is obvious that $\phi$ is well-defined. We proceed by constructing a function $\psi:\mc{L'}\rightarrow \mc{L}$, so let $\lambda\in \mc{L'}$. By taking the scalar product with $e$ on the defining equation in \eqref{eq:L'}, we see that $\la \lambda_{v(s)},e\ra \in \{0,1\}$. So by definition, there is a set $\{s_j\}_{j\in [k]}\subseteq[q]$ such that $\la\lambda_{v(s_j)},e\ra =1$ for all $j\in [k]$. Now define vectors $\{\lambda^j\}_{j\in [k]}$ according to
\begin{equation}
 \lambda^j_{v(s)}:=\begin{cases} \lambda_{v(s)} & \text{if } s=s_j, \\ 0& \text{else}  \end{cases} \quad  \forall s\in [q],\quad \forall j\in [k]
\end{equation}
and set $\psi(\lambda):=\{\lambda^j\}_{j\in [k]}$. It is easy to check that $\psi$ is well-defined. Now for $\lambda\in \mc{L'}$ one has

\begin{equation}
\phi(\psi(\lambda))_{v(s)}=\sum_{j\in [k]} \lambda^j_{v(s)} = \begin{cases} \lambda_{v(s)} & \text{if } s\in \{s_j\}_{j\in [k]}, \\ 0=\lambda_{v(s)} & \text{else,}  \end{cases}
\end{equation}
which shows $\phi \circ \psi = id_\mc{L'}$.
For $\{\lambda^j\}_{j\in[k]} \in \mc{L}$, let $\lambda=\phi(\{\lambda^j\}_{j\in[k]})$. Then we can choose $\{s_j\}_{j\in [k]}$ such that 
\begin{equation}
1=\la \lambda_{v(s_j)},e\ra = \sum_{j'\in [k]} \la \lambda^{j'}_{v(s_j)},e\ra = \la \lambda^{j}_{v(s_j)},e\ra
\end{equation}
by Lemma \ref{lem:orthogonal}, which shows $\psi \circ \phi = id_\mc{L}$.
\end{proof}

We are now prepared to restate variant \eqref{eq:membership_constraint} of (R1) as the following, symmetry free polynomial optimization problem.

\begin{subequations}\label{eq:SecondReformulation}
\begin{align}
\text{(R2)}\quad  \min_{\lambda}   \sum_{i\in[n]} & \la \lambda_i, W_i \lambda_i\ra	 \\
\text{s.t.} \enspace
 \lambda_*  &\in k\cdot \Delta^m,  & \la (\lambda_*)_{v(s)},e\ra (\lambda_*)_{v(s)} &= (\lambda_*)_{v(s)} & \forall s\in [q],\\
 \lambda_i &\in\Delta^{m}_\Omega, & \lambda_i &\leq \lambda_* & \forall i\in[n].
\end{align}
\end{subequations}
\begin{corollary}\label{cor:R2equivalence}
(R2) is equivalent to finding the optimal separated solution of~$(R1)$. In particular, if $\mc{P}$ in \eqref{eq:triangulation} is separating, then both problems are equivalent.
\end{corollary}

\section{SDP Relaxations}\label{sec:relaxation}
In this section we will exploit available theoretical results to approximate (R2) 
by a hierarchy of conic programs (R2)[t]. After some simplifications of (R2)[1], we will also introduce a relaxation of (R2)[1] which is faster to solve.

\subsection{The Hierarchy $(R2)[t]$}

In order to apply results from section \ref{sec:sdp_relax} to (R2), we have to give an explicit list of polynomial inequalities. To this end, we will use the following system, where each coordinate corresponds to one polynomial (in)equality:

\begin{subequations}\label{eq:ConstraintReformulation}
\begin{align}
\min_{\lambda}   \sum_{i\in[n]}  \la \lambda_i, &W_i \lambda_i\ra\quad s.t.	 \label{eq:sys1}\\
     \lambda_*       & \geq \lambda_i,     &                \lambda_i & \geq 0    & \forall i\in [n]\\
 \la \lambda_*, e\ra &   =      k,         &     \la \lambda_i , e\ra &   =  1    & \forall i\in [n]\\
 \la (\lambda_*)_{v(s)}, e\ra (\lambda_*)_{v(s)} &= (\lambda_*)_{v(s)} & \forall s & \in [q]\\
  (\lambda_i)_{v(s)} (\lambda_i)_{v(t)}^\T &= 0 & \forall i &\in [n],&\enspace\forall s\neq t\in [q]\label{eq:sys2}\\
 (\lambda_*)_{v(s)} (\lambda_i)_{v(s)}^\T &= (\lambda_i)_{v(s)} (\lambda_i)_{v(s)}^\T & \forall  i &\in [n],&\enspace\forall s\in [q]\label{eq:sys3}
\end{align}
\end{subequations}

One can easily check that \eqref{eq:sys1}-\eqref{eq:sys2}  is a reformulation of (R2). Additionally, we add the redundant equations \eqref{eq:sys3} implied by Lemma \ref{lem:orthogonal} since they have low degree and directly reduce the number of moments we have to consider in LMM. For $t\geq 1$, we can therefore construct the hierarchy \eqref{eq:LMM} accordingly, where we will call the optimization problem corresponding to the $t$-th step in the hierarchy as (R2)[t].

%
%
%
%
%

\subsection{Simplifying (R2)[1]}\label{sec:Simplifying}

In practice, we can only compute $(R2)[t]$ for small values of $t$. In this section, we will investigate $(R2)[1]$ in more detail and show that it can be simplified to get a smaller formulation that can be solved with current SDP solvers. To this end, we will first explicitly write down an SDP-representation of $(R2)[1]$ where we use the notation
\begin{equation}
M_1(y)=\bpm 
1         & \lambda_1^\T & \cdots & \lambda_n^\T & \lambda_*^\T \\
\lambda_1 & \Lambda_{11} & \cdots & \Lambda_{1n} & \Lambda_{1*} \\
\vdots    & \vdots       & \ddots & \vdots       & \vdots       \\
\lambda_n & \Lambda_{n1} & \cdots & \Lambda_{nn} & \Lambda_{n*} \\
\lambda^* & \Lambda_{*1} & \cdots & \Lambda_{*n} & \Lambda_{**} \\
\epm\succeq 0
\end{equation}
for the moment matrix involved. 

An important observation is that each polynomial in \eqref{eq:ConstraintReformulation} belongs to $\R[\lambda_i, \lambda_*]$ for some $i\in [n]$ and that these sets satisfy the running intersection property. Using results about sparse representations \cite[Section 8.1]{lasserre15}, this means we can ignore the matrices $\Lambda_{ij}$ for $i\neq j\in[n]$ and replace $M_1(y)\succeq 0$ by the collection of much smaller submatrices 
\begin{equation}
M_1(y| i):=\bpm 
1         & \lambda_i^\T & \lambda_*^\T \\
\lambda_i & \Lambda_{ii} & \Lambda_{i*} \\
\lambda^* & \Lambda_{*i} & \Lambda_{**} \\
\epm \succeq 0
\end{equation}
to get the reduced formulation
\begin{subequations}\label{eq:R2[1]}
\begin{align}
(R2)[1]&\quad \min_{\lambda}  \sum_{i\in[n]} \la W_i, \Lambda_{ii}\ra \\
\text{s.t.} \quad 
& \la \lambda_*, e\ra = k,\qquad \qquad\quad \Lambda_{**}e = k\lambda_*,\\
& (\Lambda_{**})_{v(s)} e = (\lambda_*)_{v(s)}   \qquad\qquad\qquad\;\;\,  \forall s  \in [q],\\
& \enspace \left.  
\begin{array}{rlrl}
(\Lambda_{i*})_{v(s)}&=(\Lambda_{ii})_{v(s)} & \quad\quad  \forall s & \in [q],\\
 \la \Lambda_{ii} ,\Omega\ra  &= 0, &  \\
 \la \lambda_i , e\ra &=1,          & \Lambda_{ii}e &=  \lambda_i, \\
 \Lambda_{*i}e &=  \lambda_*,       & \Lambda_{i*}e &= k\lambda_i, \\
\lambda_*&\geq \lambda_i \geq 0,    &  \Lambda_{**} \geq \Lambda_{*i} \geq \Lambda_{ii} & \geq 0,\\
  M_1(y|i) &\succeq 0
\end{array}
\right\rbrace \forall i\in [n].
\end{align}
\end{subequations}

Fortunately, we can also discard the linear monomials $\lambda_i$ and $\lambda_*$ with the help of the following Lemma.

\begin{lemma}\label{lem:miraculous_row_sum}
Consider a matrix $\Lambda\succeq 0$ and a vector $a$. Let $a^\T \Lambda a=\nu$ and define $\lambda:=\Lambda a$. Then 
\begin{equation} \label{eq:miraculous_row_sum_condition}
\nu \Lambda\succeq \lambda\lambda^\T \text{ or equivalently } \bpm \nu & \lambda^\T\\ \lambda & \Lambda \epm\succeq 0.
\end{equation}
\end{lemma}
\begin{proof}
Since $\Lambda\succeq 0$, there is $L$ such that $\Lambda=L^\T L$ and consequently $\nu=a^\T \Lambda a= \|La\|_2^2$. Then for arbitrary $x$ we have
\begin{align*}
x^\T (\nu \Lambda-(\Lambda a)(\Lambda a)^\T)x &=\nu\cdot \langle Lx, Lx\rangle - \langle Lx , La \rangle^2\\
&=\|La\|^2_2 \cdot \|Lx\|^2_2 - |\langle Lx, La\rangle|^2 \geq 0
\end{align*}
where the last inequality is the Cauchy-Schwarz inequality. The equivalent second formulation of \eqref{eq:miraculous_row_sum_condition} follows from the Schur Complement Theorem.
\end{proof}

In terms of \eqref{eq:R2[1]}, using Lemma \ref{lem:miraculous_row_sum} on 
\begin{equation}
\bpm \Lambda_{ii} & \Lambda_{i*}\\ \Lambda_{*i} & \Lambda_{**}\epm \succeq 0
\end{equation}
with the vector $a=\bpm e& 0\epm^\T$ already yields the condition $M_1(y|i)\succeq 0$, which means we can discard the linear monomials in the following equivalent formulation

\begin{subequations}\label{eq:R2'[1]}
\begin{align}
(R2')[1]&\quad \min_{\lambda}  \sum_{i\in[n]} \la W_i, \Lambda_{ii}\ra \\
\text{s.t.} & \enspace \left. 
\begin{array}{rlrl}
(\Lambda_{i*})_{v(s)}&=(\Lambda_{ii})_{v(s)} & \quad \forall s & \in [q],\\
(\Lambda_{**})_{v(s)} e & = \left( \Lambda_{*i} e\right)_{v(s)}& \quad \forall s & \in [q],\\
 \la \Lambda_{ii} ,\Omega\ra  &= 0,\\
 k\Lambda_{ii}e &=  \Lambda_{i*}e, &  \la \Lambda_{ii},J\ra &=  1 , \\
 k\Lambda_{*i}e &=  \Lambda_{**}e, &  \la \Lambda_{*i},J\ra &=  k, \\
\Lambda_{**} \geq \Lambda_{*i} &\geq \Lambda_{ii}  \geq 0, &
 \bpm \Lambda_{ii} & \Lambda_{i*}\\ \Lambda_{*i} & \Lambda_{**}\epm  &\succeq 0
\end{array}
\right\rbrace \forall i\in [n].
\end{align}
\end{subequations}
This reformulation uses $n$ SDP matrices of dimension $2m=2q(d+1)$, which is still very limiting. Note, however, that since $\la \Lambda_{ii}, \Omega \ra =0$, there is still a huge sparsity pattern in the blockdiagonal $\Lambda_{ii}$, which is not properly exploited.

For this reason, we propose to relax $(R2')[1]$ by dropping all variables in $\Lambda_{i*}$ and $\Lambda_{**}$ that do not belong to the blockdiagonal structure induced by $\Omega$. Effectively, this means we lose information of entries in $\Lambda_{**}$ that only have an indirect impact on $\Lambda_{ii}$. This turns each SDP constraint in $(R2')[1]$ into $q$ separate SDP constraints
\begin{equation}
\bpm (\Lambda_{ii})_{v(s)} & (\Lambda_{i*})_{v(s)}\\ (\Lambda_{*i})_{v(s)} & (\Lambda_{**})_{v(s)}\epm  \succeq 0 \quad \forall s\in [q]
\end{equation}
of size $2(d+1)$, which is again much smaller. However, since 
\begin{equation}
(\Lambda_{i*})_{v(s)}=(\Lambda_{ii})_{v(s)},
\end{equation} this is equivalent to
\begin{equation}
(\Lambda_{**})_{v(s)}\succeq (\Lambda_{ii})_{v(s)} \quad \forall s\in [q],
\end{equation}
since $\bpm A & A\\ A & B\epm\succeq 0$ is equivalent to $B\succeq A\succeq 0$ as a consequence of the Schur complement theorem.

Formally, we end up with the following relaxation of $(R2)[1]$, which we will call
\begin{subequations}\label{eq:R2"[1]}
\begin{align}
(R2'')[1]&\quad \min_{\lambda}  \sum_{i\in[n]} \la W_i, \Lambda_{ii}\ra \\
\text{s.t.} & \enspace \left. 
\begin{array}{rlrl}
 \la \Lambda_{ii} ,\Omega\ra &= 0, & \la \Lambda_{ii},      J\ra &= 1, \\
 (\Lambda_{**})_{v(s)} &\succcurlyeq (\Lambda_{ii})_{v(s)}  \succcurlyeq 0 & \forall s&\in[q]
\end{array}
\right\rbrace \forall i\in [n],\\
& \quad \sum_{s\in [q]}\la (\Lambda_{**})_{v(s)},J\ra =  k.
\end{align}
\end{subequations}
Note that the last constraint follows from $(R2')[1]$ as
\begin{equation}
k = \la \Lambda_{i*},J\ra = \sum_{s\in [q]} e_{v(s)}^\T (\Lambda_{*i}e)_{v(s)}
= \sum_{s\in [q]} e^\T (\Lambda_{*i})_{v(s)} e.
\end{equation}

\subsection{The variant (R3)[t]}

Instead of introducing $\lambda_*$ as the sum of optimal parameters $\lambda^j$ in \eqref{eq:*definition}, we might directly work with the corresponding moment sequences. Letting $y^j\in \mc{N}^*(\Delta_\Omega^d)$ denote the moment sequence of $\lambda^j$, we can define
\begin{equation}
y_* = \sum_{j\in [k]} y^j \in \mc{N}^*(\Delta_\Omega^d),\quad (y_*)_0=k,
\end{equation}
to get the implication
\begin{equation}
y\in \{y^j\}_{j\in [k]} \quad \Rightarrow \quad y, y^*-y\in \mc{N}^*(\Delta_\Omega^d),\quad y_0 = 1
\end{equation}
as a weaker alternative to \eqref{eq:containment}.
Applying LMM on the set $\mc{N}^*(\Delta_\Omega^d)$ then leads to the hierarchy 
\begin{subequations}\label{eq:R3}
\begin{align}
\text{(R3)[t]}\quad  \min_{y } & \sum_{i\in[n]} L_{y_i}(W_i)\\
\text{s.t.} & \quad
\begin{array}{rlcrl}
y_i & \in \mc{N}^*_t (\Delta_\Omega^m),&& (y_i)_0 &= 1,\\
y^*-y_i &\in \mc{N}^*_t (\Delta_\Omega^m),&& y^*_0-(y_i)_0 &= k-1
\end{array}
\quad \forall i\in[n],
\end{align}
\end{subequations}
where it can be shown that for $t=1$, this coincides with $(R2'')[1]$ after properly reformulating \eqref{eq:R2"[1]}.

$(R3)[1]$ uses $2nq$ SDP constraints of the rather small size $d+1$ and is only weakly coupled, so that parallel computing schemes can be efficiently used to solve the relaxation for problems of moderate paramaters $(n,q,d)$. 

\begin{remark}
Model $(R3)[1]$ coincides with problem $(21)$ given in \cite{Silvestri15} in the respective setting. 
\end{remark}

\section{Rounding}\label{sec:rounding}

%

Given a solution $(y^*,y_i)$ to $(R3)[1]$, a rounding procedure has to determine 
a proper partition of $[n]$. To do this, we will use the information provided by the convex relaxation and a k-center clustering algorithm, as detailed next.

\begin{definition}[k-center Clustering] 
Given a set $K=\{x_i\}_{i\in [n]}\subseteq \mathbb{R}^d$ and a norm $\|\cdot \|$, the k-center clustering problem is defined as
\begin{equation}
\mc{C}^{\|\cdot \|}_\infty(K,k):=\min_{ C\subseteq K, |C|=k}\quad \max_{i\in[n]}\quad \min_{y\in C} \, \|y-x_i\|.\label{eq:Metric Clustering}
\end{equation}
\end{definition}

Since this problem is NP-hard, we need to use a heuristic instead, which should be insensitive to initializations. This can be achieved as follows.

\begin{theorem}[Approximating k-Center Clustering \cite{hochbaum85}]
For $d>2$, achieving an approximation ratio for \eqref{eq:Metric Clustering}  better than $2$ is NP-hard. A $2$-approximation is given by the following deterministic algorithm.

\begin{algorithm}[H]\label{alg:Farthest_Point_Clustering}
\caption{$\FPC(K,k,\|\cdot \|)$  (Farthest Point Clustering)} 
\SetKw{NOT}{not}
\SetKw{OR}{or}

\SetKwFunction{F}{F}

\KwData{Data $K=\{x_i\}_{i\in [n]}\subseteq \mathbb{R}^d$, norm $\|\cdot \|$, $k\in [n]$}
\KwResult{Centers $C\subseteq \{x_i\}_{i\in [n]}, |C|=k$}

$B \leftarrow \infty$, $C \leftarrow \emptyset$\;
\For{$i\in[n]$}{
	$C_i\leftarrow \emptyset$,	$c_1 \leftarrow x_i$\;	
	\For{$j\in [k]$}{
		$C_i \leftarrow C_i\cup \{c_j\}$, $c_{j+1}\leftarrow \argmax_{x\in K}\min_{y\in C_i} \|x-y\|$\;		
	}
	\If{$\min_{y\in C_i} \|c_{k+1}-y\| <B$}{
		$B\leftarrow \min_{y\in C_i} \|c_{k+1}-y\|$, $C\leftarrow C_i$\;
	}

}

\Return{$C$}\;
\end{algorithm}
\end{theorem}

Algorithm \ref{alg:Farthest_Point_Clustering} greedily builds the set of cluster centers $C$ by iteratively choosing those points which are farthest away from all prior centers. As initialization, every point is chosen as the first cluster center once and the best overall result is kept as the output of the algorithm.

\noindent Our rounding procedere can now be described as follows, where $W=(W_i)_{i\in[n]}$ denotes the objective function.
\vspace*{3 mm}

\begin{algorithm}[H]\label{alg:Rounding}
\caption{k-Cluster Rounding} 
\KwData{Objective $W$ and solution $(y^*, y_i)$ of $(R3)[1]$.}
\KwResult{Solution to \eqref{eq:ProblemFormulation} of value $\rnd(W)$.}
Extract the second order moments $\Lambda_i$ of $y_i$ according to section \ref{sec:Simplifying}\;
set $\lambda_i = \Lambda_i e$ for all $i\in [n]$\;
set $U$ equal to the partition of the minimizer of $\FPC(\{\lambda_i\}_{i\in[n]},k,\ell_1)$\;
for fixed $U$, compute optimal centers $\{x_j\}_{j\in[k]}$ in \eqref{eq:ProblemFormulation}\; \label{line:fixed_partition}
set $\rnd(W)$ to the objective value of $(U,\{x_j\}_{j\in[k]})$ in \eqref{eq:ProblemFormulation}\;
\Return{$(U, \{x_j\}_{j\in[k]},\rnd(W))$ }\;
\end{algorithm}

Algorithm \ref{alg:Rounding} clusters the $\lambda$-representations from $(R3)[1]$ according to their $\ell_1$-norm to construct the assignment matrix $U$. Afterwards, the actual centers are computed as the analytic solution to \eqref{eq:ProblemFormulation} with fixed assignments.

\section{Extensions}\label{sec:extention}

We next comment on the choice of $(V,\Omega)$ in section \ref{sec:k-clustering} and then indicate a generalization of $\mc{P}$ to semialgebraic sets and more general objective functions. Even though the variants here are presented in terms of the original problem \eqref{eq:ProblemFormulation} and (R1), the machinery of section \ref{sec:notation} can be used in a straightforward way to process the modifications for (R2) and (R3[t]).

Recall that $\Omega$ is assumed to separate the individual parametrization of the local parameters $\{\lambda_i\}_{i\in[n]}$ in $V$ in the preceding sections. While this guarantees that $(R2)$ is a reformulation of $(R1)$, we might also consider relaxing this constraint by breaking up the blockdiagonal structure. While we lose much of the underlying theory this way, we may also gain a speed up heuristically.

\subsection{Unique Columns in $V$}\label{sec:variants}

As mentioned in Remark \ref{rem:redundantColumns}, we do not assume the columns of $V$ to be unique in Section \ref{sec:k-clustering}. This is done to ensure the blockdiagonal structure of $\Omega$, but can be relaxed.

\begin{example}
Let $\mc{P}=P_1\cup P_2$ where 
$V_1 = \bsm -1 & 0\\ 0 & 1\esm$ 
and 
$V_2 = \bsm  0 & 1\\ 1 & 0\esm$ 
to get 
$V = \bsm -1 & 0 & 0 & 1\\ 0 & 1 & 1 & 0\esm$ 
and 
$\Omega = \bsm 0 & J_2 \\ J_2 & 0\esm$.
Instead, we might as well use
$V = \bsm -1 & 0 & 1\\ 0 & 1 & 0\esm$ 
and 
$\Omega = \bsm 
0 & 0 & 1 \\
0 & 0 & 0\\
1 & 0 & 0\esm$.

It is important to note that while this reduces $m$ and therefore improves the running time for solving $(R2[t])$ and $(R3[t])$, it also has a negative impact on the quality of the solutions if $P_1$ and $P_2$ are each assumed to contain a different local optimizer $x_j$.

In particular, the new formulation gives a ``discount'' on using the common vertex $\bsm 0 \\ 1 \esm$ so that more weight is assigned to $\bsm -1 \\ 0 \esm$ and $\bsm -1 \\ 0 \esm$. Consequently, the individual $\lambda_i$ will have a larger spread, as can be seen by comparing Figures~\ref{fig:euclidean_choice} and \ref{fig:spread}.
This is because the reduced formulation only needs to increase one entry of $\lambda_*$ to use $\bsm 0 \\ 1 \esm$ in convex combinations of \emph{both} $P_1$ and $P_2$, while one entry for \emph{each} $P_1$ and $P_2$ had to be increased in the original formulation for the same effect.

Note that this is desirable if the union $P_1 \cup P_2$ is meant to only contain a single local optimizer $x_j$.
\end{example}

\subsection{$\mc{P}$ as $\sigma$-Skeleton of Arbitrary Polytopes}

In order to suppress certain combinations of vertices of $\mc{P}$ to appear simultaneously in a parametrization, $\Omega$ can be extended to sum up the corresponding moments as well. Conversely, we can start with an arbitrary polytope $P$ and remove all faces whose dimension exceeds $\sigma\in \N$ to describe the $\sigma$-skeleton of $P$.

\begin{definition}
Let $\mc{P}$ be a single polytope, $\mc{A}_\mc{P}$ the adjacency matrix of the graph of $\mc{P}$ and $\Omega=J_m-I_m-\mc{A}_{P}$, so that $\Omega$ encodes all pairs of vertices whose connecting line segment passes through the interior of $\mc{P}$. Then the $\sigma$-skeleton $skel_\sigma(\mc{P})$ of $\mc{P}$ can be formally defined as

\begin{equation}
skel_\sigma(\mc{P}) := \{ x=V\lambda \colon \lambda\in \Delta^m_\Omega,\quad \|\lambda\|_0\leq \sigma+1 \},
\end{equation}
which is the union of all faces of $\mc{P}$ of dimension at most $\sigma$.
\end{definition}

\begin{remark}
The set $\{\lambda\in \Delta^m\colon \|\lambda \|_0 \leq \sigma\}$ can be described by adding the equation
\begin{equation}
\sum_{S\subseteq [m]\colon |S|\geq \sigma+1} y_S =0
\end{equation}
which can be incorporated into the equation given by $\Omega$. Of course, this will become quickly impractical since it requires $t>\sigma$ in $(R2[t])$ or $(R3[t])$ to work.
\end{remark}

\begin{example}
The unit square $C_2$ is given as the convex hull of $V=\bsm 0 & 1 & 0 & 1\\ 0 & 0 & 1 & 1\esm$ and $skel_1(C_2)$ consists of $4$ line segments. Choosing $\mc{P}=skel_1(C_2)$ we need $4$ simplices and consequently $m=8$ vertices for the approach in section \ref{sec:k-clustering}, but adding the sparsity constraint $\|\lambda\|_0\leq 2$ in $(R2[2])$ or $(R3[2])$ allows us to use each vertex only once to end up with $m=4$.
\end{example}

It would be interesting to investigate low-degree polynomials as approximations of sparsity constraints.

\subsection{Semialgebraic $K$}\label{sec:semialgebraic_K}

While the preceding sections worked on constraints regarding the parametrization, polynomial constraints on the local optimizers $\{x_j\}_{j\in[k]}$ can be incorporated into the framework laid out in section \ref{sec:notation} as well. In fact, for $x=V\lambda$ any polynomial expression $f(x)$ can be easily turned into a polynomial expression in $\lambda$ \emph{of the same degree} by setting $f'(\lambda)=f(V\lambda)$. So, in general, the feasible space $K$ can be assumed as a compact basic semialgebraic set - one only needs to cover this set by polytopes to use our approach. In particular, there is no need to \emph{approximate} $K$ by polytopes as long as $K$ is \emph{covered} by them.

However, depending on the geometry of the underlying set, it may be harder to choose a separating triangulation $\mc{P}$.

\begin{example}
Assume each local optimizer $x_j\in\R^d$, $j\in[k]$ should be normalized by $\|x_j\|_2=1$. Squaring this condition gives the quadratic polynomial equation $x_j^\T x_j=1$. Now substituting $x_j=V\lambda^j$ yields again a quadratic constraint ${\lambda^j}^\T \big(V^\T V\big) \lambda^j=1$.
\end{example}

\begin{example}
Assume each local optimizer $x_j\in\R^4$, $j\in[k]$ should encode a vectorized orthogonal $2\times 2$ matrix $X_j$. This yields four quadratic equations, one for each entry of $X_j X_j^\T= I_2$. Denoting them by 
$x_j^\T Q_l x_j = q_l$, substituting $x_j=V\lambda^j$ yields again quadratic constraints ${\lambda^j}^\T \big(V^\T Q_l V\big) \lambda^j=q_l$.
\end{example}

As a caveat, however, we point out that even though the hierarchy $(R2[t])$ will converge towards feasibility in the actual sets, lower levels may only give crude approximations.

\subsection{Clustering Varieties}

As already mentioned in the introduction, our approach can be easily extended from affine subspaces to the case of varieties. We simply replace $A_i x_j - b_i$ in \eqref{eq:ProblemFormulation} with $F_i(x_j)$, where $F_i \in \R[x]$ is an arbitrary multivariate polynomial and encodes the variety $\var2( \|F_i\|^2_2 )$. Following Section \ref{sec:k-clustering}, we may replace $x_j$ by $V\lambda_j$ and homogenize $\| F_i(V\lambda_j)\|^2_2$ using $\la \lambda_j, e\ra = 1$ to end up with a variant of (R2) where the objective function has been replaced. The results from Section \ref{sec:relaxation} follow according to this replacement, with the additional constraint that we can only consider $(R2)[t]$ or $(R3)[t]$ for values of $t\geq \max_{i\in[n]}\deg(F_i)$. Of course, we can still use our rounding heuristic presented in section \ref{sec:rounding}.

\subsection{Regularization with respect to $k$}

It should be noted that our relaxation (R2) never explicitly depends on the number $k$ apart from the constraint 
\begin{equation}
\la \lambda_{*},e\ra =  k
\end{equation}
in \eqref{eq:ConstraintReformulation}. It is therefore possible to treat $k$ throughout as a variable and include a weighted $k$ in the objective function in order to dynamically search for the number of clusters.

\section{Experiments}\label{sec:experiments}

All the examples in this section were carried out in Matlab using the SDPT3 package \cite{TTT96,TTT03}. 

\subsection{Euclidean Clustering}
By choosing $A_i=I$ in \eqref{eq:ProblemFormulation} we recover the classical problem of Euclidean clustering for the points $\{b_i\}_{i\in[n]}\subseteq \R^d$. For this problem, it is well known that 
\begin{equation}
\mc{P}\supseteq \conv(\{b_i\}_{i\in [n]})
\end{equation}
will contain all optimal parameters \cite{KMeans-SDP-07}. In particular, for the triangulation assumption it suffices that $\mc{P}$ covers a box which includes all $\{b_i\}_{i\in [n]}$, which can be easily extracted.\\

We can use Euclidean clustering to get a better intuition of how the relaxation works. Regarding the choice of $\mc{P}$, consider Figure \ref{fig:euclidean_choice}. Using any simplex containing all the points is the coarsest approximation but yields useless results, since each local estimate $V\lambda_i$ can be chosen as $b_i$. 

In view of corollary \ref{cor:R2equivalence}, the algorithm will perform best if the triangulation is separating, so that the cluster centers are separated by the polytopes in $\mc{P}$. This suggests that there should be at least $k$ polytopes, and that an oversegmentation removes the need knowing the ground truth, as can be observed in Figure \ref{fig:euclidean_choice}.


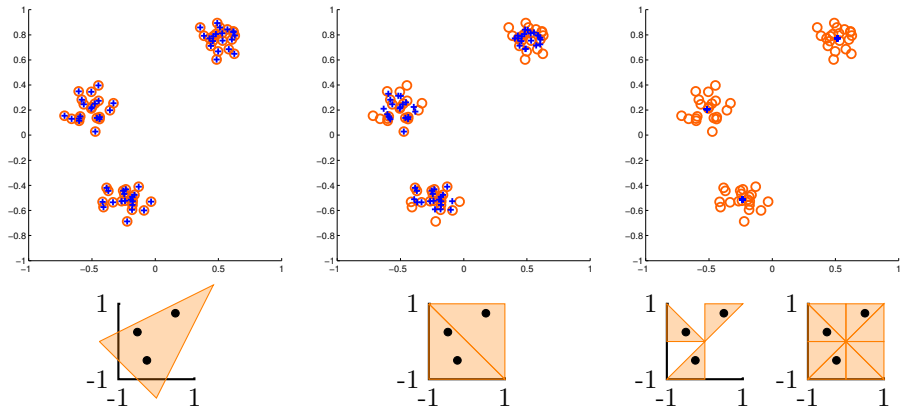
\begin{figure}[!h]
\begin{tabular}{c c c c}

\includegraphics[width=0.3\textwidth]{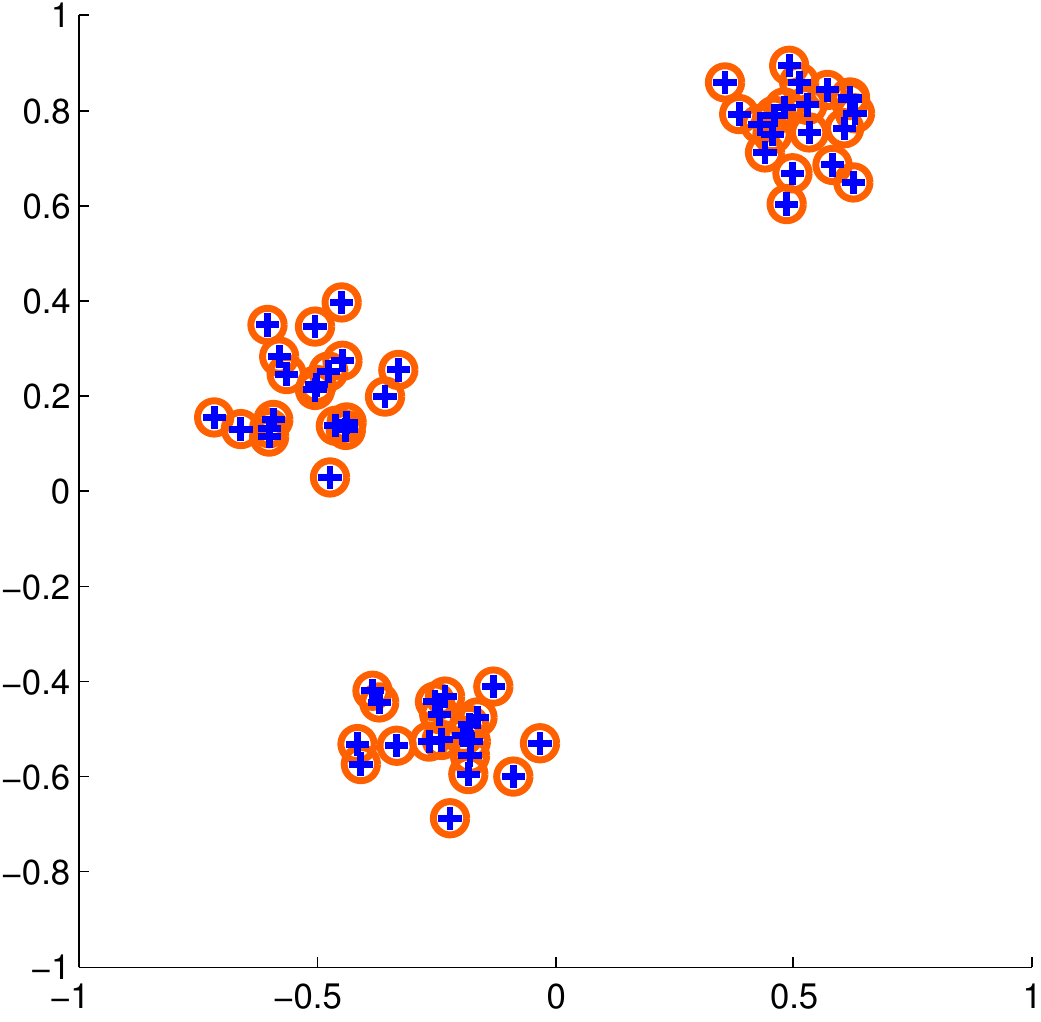} 
&
\includegraphics[width=0.3\textwidth]{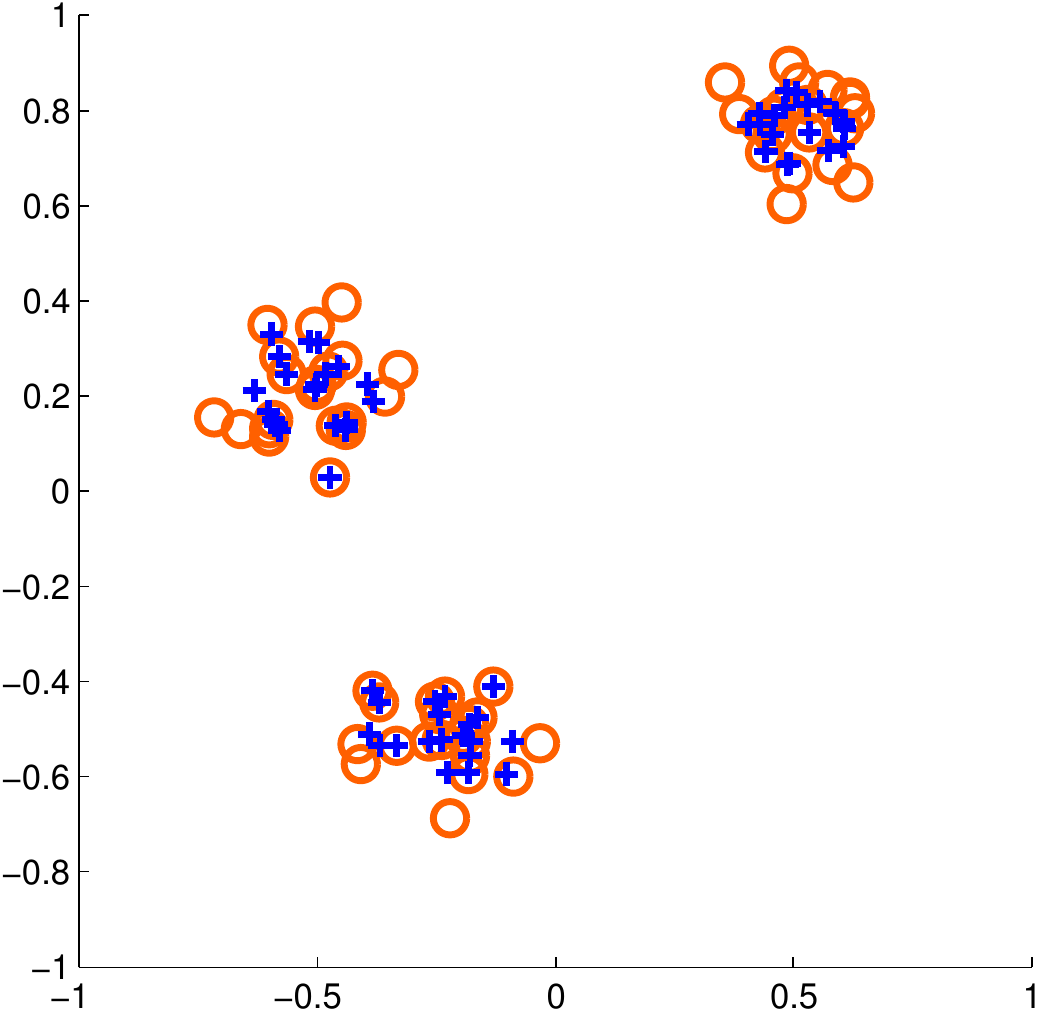}
&
\includegraphics[width=0.3\textwidth]{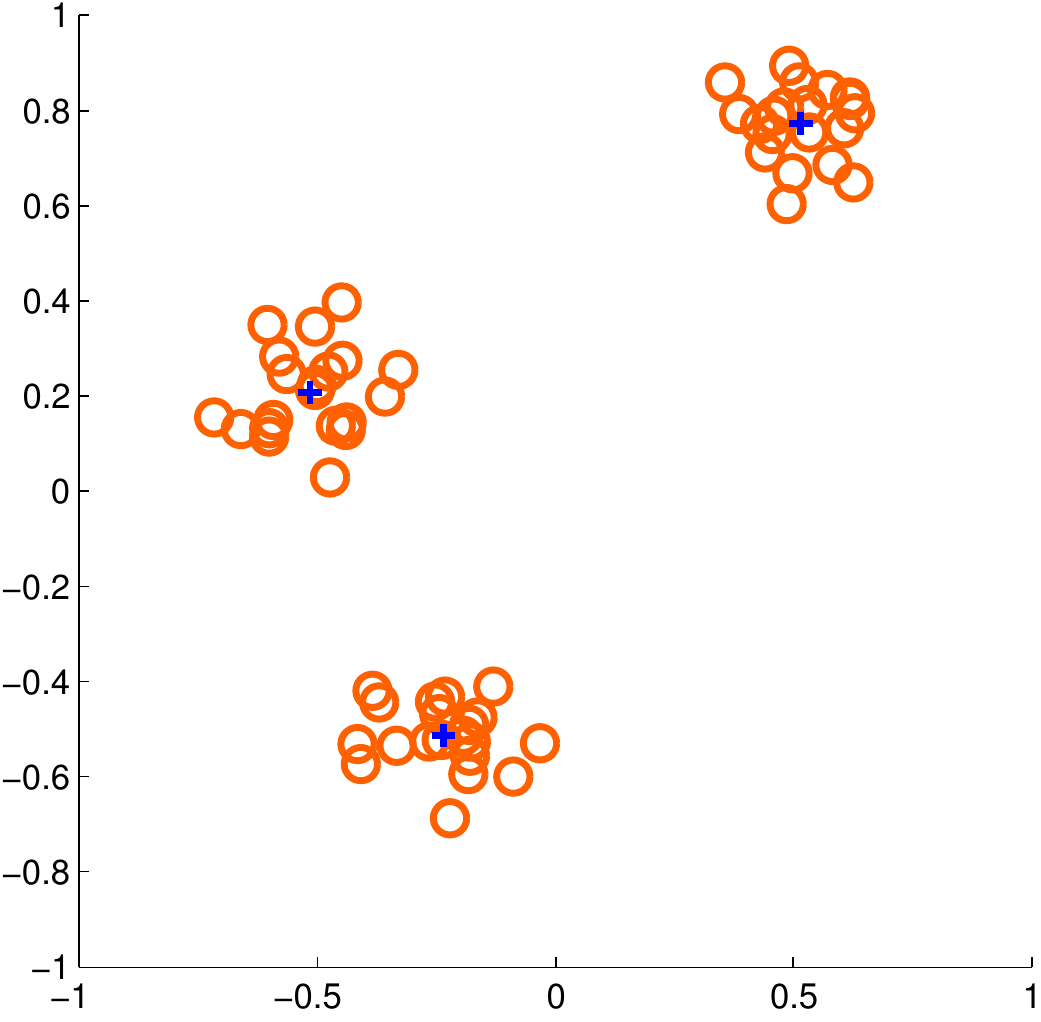} \\

\begin{tikzpicture}[scale=0.5]
\draw [thick] (-0.95 ,1) -- (-1, 1) -- (-1,-1) -- (1,-1) -- (1,-0.95);
\node [below] at ( -1 , -1) {-1};
\node [below] at (  1 , -1) { 1};
\node [left]  at ( -1 , -1) {-1};
\node [left]  at ( -1 ,  1) { 1};

\draw [orange, fill=orange,fill opacity=0.3] (-1.5 , 0) -- (0 , -1.5) -- (1.5 , 1.5) -- (-1.5 , 0);

\draw [fill,black] ( 0.5  ,  0.75 ) circle [radius=0.1];
\draw [fill,black] (-0.25 , -0.5  ) circle [radius=0.1];
\draw [fill,black] (-0.5  ,  0.25 ) circle [radius=0.1];
\end{tikzpicture}

&

\begin{tikzpicture}[scale=0.5]
\draw [thick] (-0.95 ,1) -- (-1, 1) -- (-1,-1) -- (1,-1) -- (1,-0.95);
\node [below] at ( -1 , -1) {-1};
\node [below] at (  1 , -1) { 1};
\node [left]  at ( -1 , -1) {-1};
\node [left]  at ( -1 ,  1) { 1};

\draw [orange, fill=orange,fill opacity=0.3] (-1 , 1) -- (-1, -1) -- (1 , -1) -- (-1, 1);
\draw [orange, fill=orange,fill opacity=0.3] (-1 , 1) -- ( 1,  1) -- (1 , -1) -- (-1, 1);

\draw [fill,black] ( 0.5  ,  0.75 ) circle [radius=0.1];
\draw [fill,black] (-0.25 , -0.5  ) circle [radius=0.1];
\draw [fill,black] (-0.5  ,  0.25 ) circle [radius=0.1];
\end{tikzpicture}

&

\begin{tikzpicture}[scale=0.5]
\draw [thick] (-0.95 ,1) -- (-1, 1) -- (-1,-1) -- (1,-1) -- (1,-0.95);
\node [below] at ( -1 , -1) {-1};
\node [below] at (  1 , -1) { 1};
\node [left]  at ( -1 , -1) {-1};
\node [left]  at ( -1 ,  1) { 1};

\draw [orange, fill=orange,fill opacity=0.3] (0 , 0) -- (-1, 0) -- (-1 ,  1) -- (0, 0);
\draw [orange, fill=orange,fill opacity=0.3] (0 , 0) -- (0, -1) -- (-1 , -1) -- (0, 0);
\draw [orange, fill=orange,fill opacity=0.3] (0 , 0) -- (0,  1) -- ( 1 ,  1) -- (0, 0);

\draw [fill,black] ( 0.5  ,  0.75 ) circle [radius=0.1];
\draw [fill,black] (-0.25 , -0.5  ) circle [radius=0.1];
\draw [fill,black] (-0.5  ,  0.25 ) circle [radius=0.1];
\end{tikzpicture}

\begin{tikzpicture}[scale=0.5]
\draw [thick] (-0.95 ,1) -- (-1, 1) -- (-1,-1) -- (1,-1) -- (1,-0.95);
\node [below] at ( -1 , -1) {-1};
\node [below] at (  1 , -1) { 1};
\node [left]  at ( -1 , -1) {-1};
\node [left]  at ( -1 ,  1) { 1};

\draw [orange, fill=orange,fill opacity=0.3] (0 , 0) -- (-1, 0) -- (-1 ,  1) -- (0, 0);
\draw [orange, fill=orange,fill opacity=0.3] (0 , 0) -- ( 0,-1) -- (-1 , -1) -- (0, 0);
\draw [orange, fill=orange,fill opacity=0.3] (0 , 0) -- ( 1, 0) -- ( 1 , -1) -- (0, 0);
\draw [orange, fill=orange,fill opacity=0.3] (0 , 0) -- ( 0, 1) -- ( 1 ,  1) -- (0, 0);

\draw [orange, fill=orange,fill opacity=0.3] (0 , 0) -- ( 0, 1) -- (-1 ,  1) -- (0, 0);
\draw [orange, fill=orange,fill opacity=0.3] (0 , 0) -- (-1, 0) -- (-1 , -1) -- (0, 0);
\draw [orange, fill=orange,fill opacity=0.3] (0 , 0) -- ( 0,-1) -- ( 1 , -1) -- (0, 0);
\draw [orange, fill=orange,fill opacity=0.3] (0 , 0) -- ( 1, 0) -- ( 1 ,  1) -- (0, 0);

\draw [fill,black] ( 0.5  ,  0.75 ) circle [radius=0.1];
\draw [fill,black] (-0.25 , -0.5  ) circle [radius=0.1];
\draw [fill,black] (-0.5  ,  0.25 ) circle [radius=0.1];
\end{tikzpicture}

\end{tabular}

\caption{
\textbf{Euclidean Clustering} with $d=2$, $k=3$, $n=60$.
Top: Circles corresponding to data points and crosses corresponding to local estimates for centers parametrized by $\lambda_i$ extracted from $(R3[1])$.
Bottom: Different Choices of $\mc{P}$.
From left to right: Minimal cover, nonseparating cover, perfect cover (based on ground truth), oversegmentation. The rounding procedure was able to recover the optimal solution implied by the right plot in each scenario.}
\label{fig:euclidean_choice}
\end{figure}


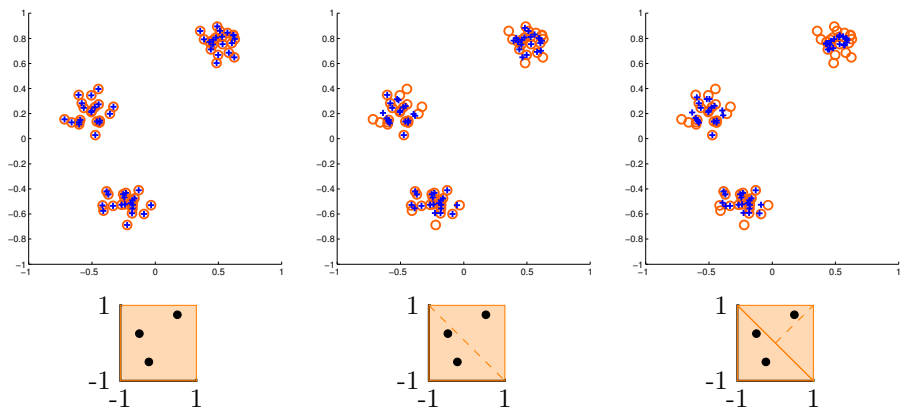
\begin{figure}[!h]
\centering
\begin{tabular}{c c c c}

\includegraphics[width=0.3\textwidth]{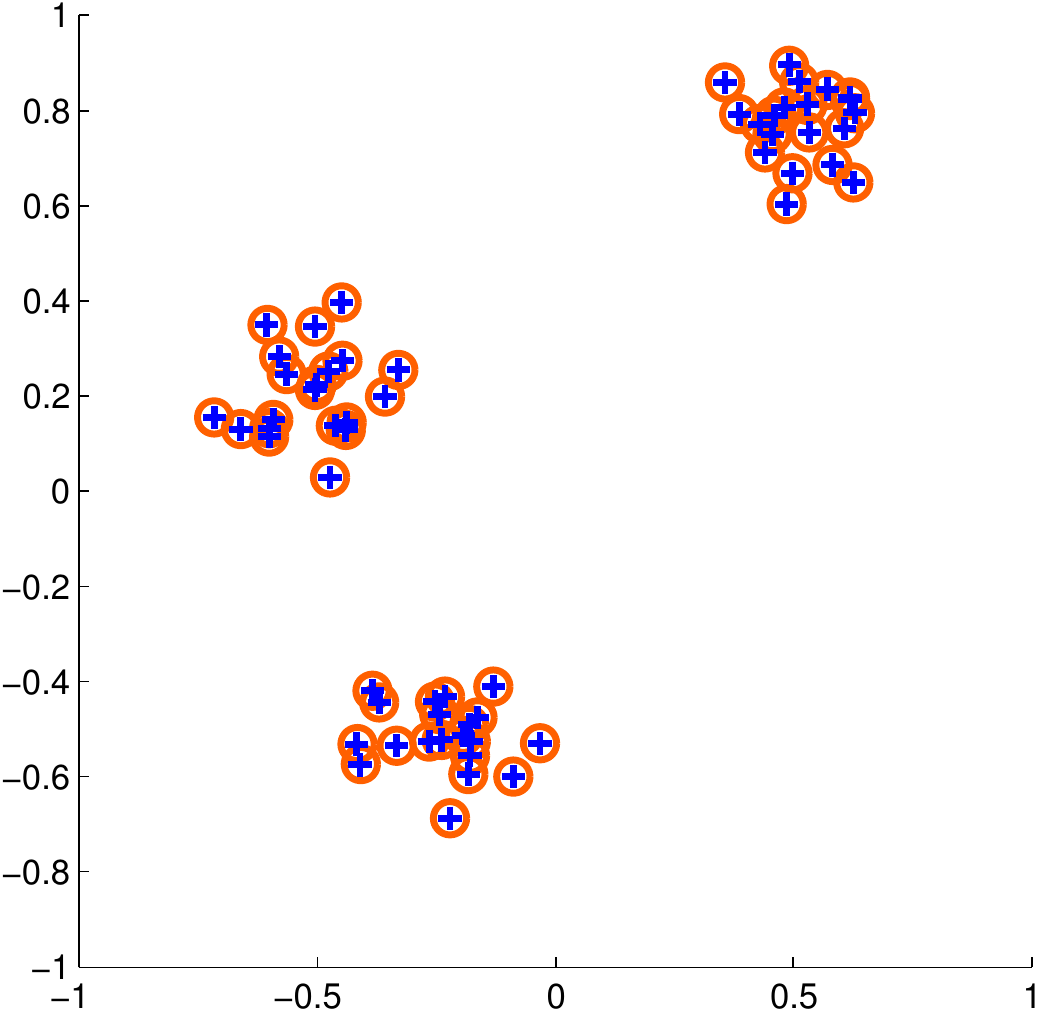}
&
\includegraphics[width=0.3\textwidth]{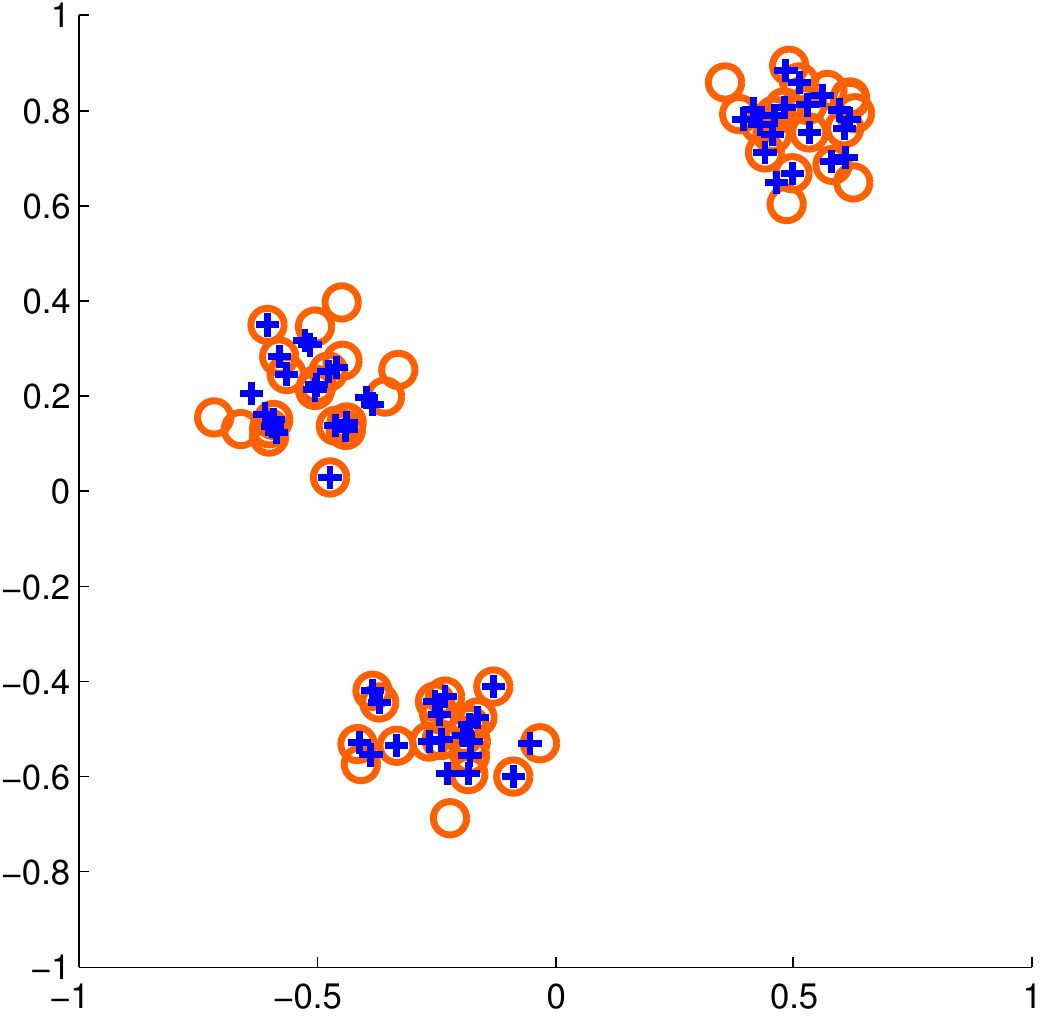}
&
\includegraphics[width=0.3\textwidth]{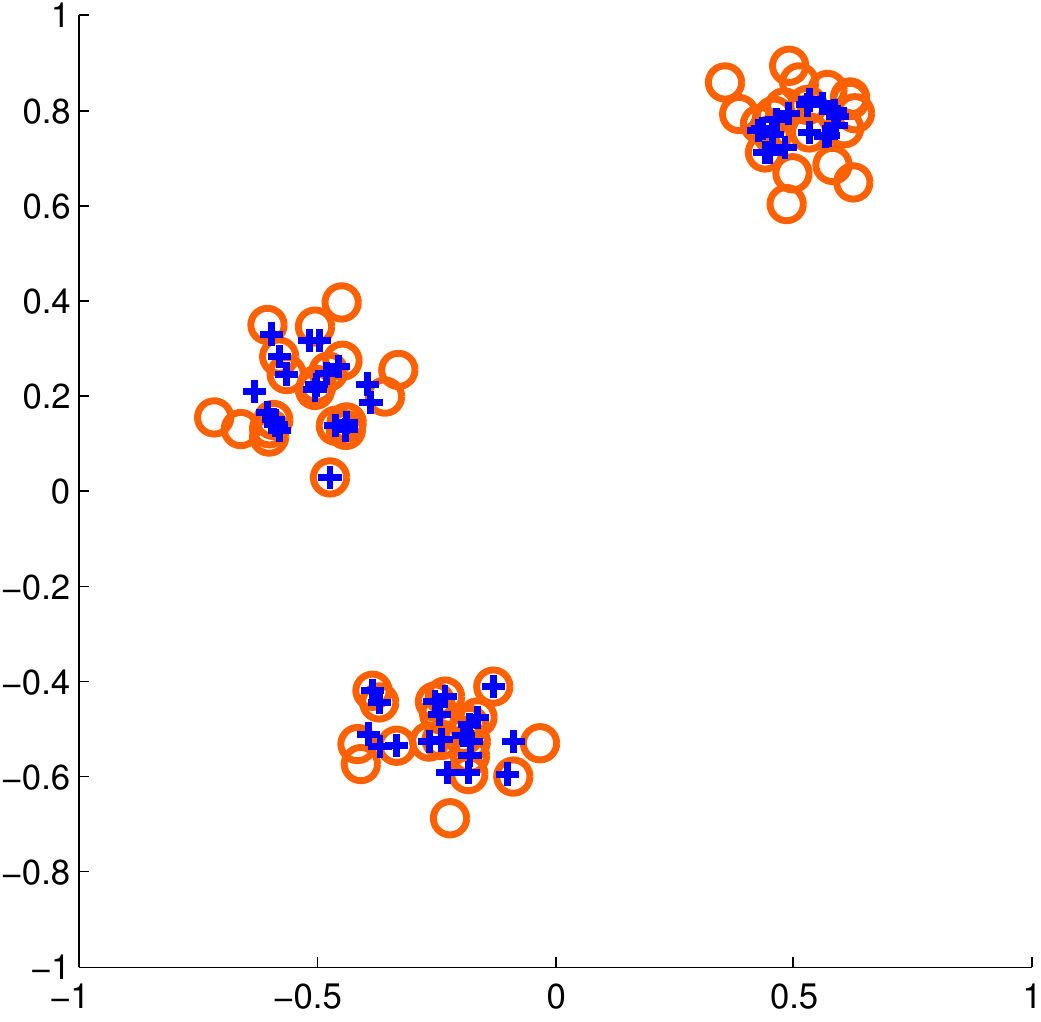} \\

\begin{tikzpicture}[scale=0.5]
\draw [thick] (-0.95 ,1) -- (-1, 1) -- (-1,-1) -- (1,-1) -- (1,-0.95);
\node [below] at ( -1 , -1) {-1};
\node [below] at (  1 , -1) { 1};
\node [left]  at ( -1 , -1) {-1};
\node [left]  at ( -1 ,  1) { 1};

\draw [orange, fill=orange,fill opacity=0.3] (-1 , 1) -- (-1, -1) -- (1 , -1) -- (1, 1) -- (-1 , 1);

\draw [fill,black] ( 0.5  ,  0.75 ) circle [radius=0.1];
\draw [fill,black] (-0.25 , -0.5  ) circle [radius=0.1];
\draw [fill,black] (-0.5  ,  0.25 ) circle [radius=0.1];
\end{tikzpicture}
&
\begin{tikzpicture}[scale=0.5]
\draw [thick] (-0.95 ,1) -- (-1, 1) -- (-1,-1) -- (1,-1) -- (1,-0.95);
\node [below] at ( -1 , -1) {-1};
\node [below] at (  1 , -1) { 1};
\node [left]  at ( -1 , -1) {-1};
\node [left]  at ( -1 ,  1) { 1};

\draw [orange, fill=orange,fill opacity=0.3] (-1 , 1) -- (-1, -1) -- (1 , -1) -- (1, 1) -- (-1 , 1);
\draw [orange, dashed, fill opacity=0.3] (-1 , 1) -- (1 , -1);

\draw [fill,black] ( 0.5  ,  0.75 ) circle [radius=0.1];
\draw [fill,black] (-0.25 , -0.5  ) circle [radius=0.1];
\draw [fill,black] (-0.5  ,  0.25 ) circle [radius=0.1];
\end{tikzpicture}
&
%
%
%
%

\begin{tikzpicture}[scale=0.5]
\draw [thick] (-0.95 ,1) -- (-1, 1) -- (-1,-1) -- (1,-1) -- (1,-0.95);
\node [below] at ( -1 , -1) {-1};
\node [below] at (  1 , -1) { 1};
\node [left]  at ( -1 , -1) {-1};
\node [left]  at ( -1 ,  1) { 1};

\draw [orange, fill=orange,fill opacity=0.3] (-1 , 1) -- (-1, -1) -- (1 , -1) -- (-1, 1);
\draw [orange, fill=orange,fill opacity=0.3] (-1 , 1) -- ( 1,  1) -- (1 , -1) -- (-1, 1);
\draw [orange, dashed, fill opacity=0.3] (0 , 0) -- (1 , 1);

\draw [fill,black] ( 0.5  ,  0.75 ) circle [radius=0.1];
\draw [fill,black] (-0.25 , -0.5  ) circle [radius=0.1];
\draw [fill,black] (-0.5  ,  0.25 ) circle [radius=0.1];
\end{tikzpicture}

\end{tabular}

\caption{\textbf{Euclidean Clustering} with $d=2$, $k=3$, $n=60$. Top: Circles corresponding to data points and crosses corresponding to local estimates for centers parametrized by $\lambda_i$ extracted from $(R3[1])$.
Bottom: Variants described in section \ref{sec:variants} for describing the feasible set $[-1,1]^2$.
From left to right: $\mc{P}$ is the union of $1,2,3$ polytopes respectively. Vertices at dashed lines are unique rows in $V$ and used in each bordering polytope.}
\label{fig:spread}
\end{figure}


Since we set up $\mc{P}$ as choice of arbitrary polytopes, we can easily restrict the feasible set in a way to force the optimal solution into specific regions. For example, by choosing each polytope in $\mc{P}$ to be a single vertex, we reduce $(R3[1])$ to an LP which aims to choose an optimal collection of locations from a discrete set of points, as can be seen in figure \ref{fig:facilities}. In this case, our experiments always returned the optimal solution.

\begin{figure}[!h]
\centering
\begin{tabular}{c c c}

\includegraphics[width=0.3\textwidth]{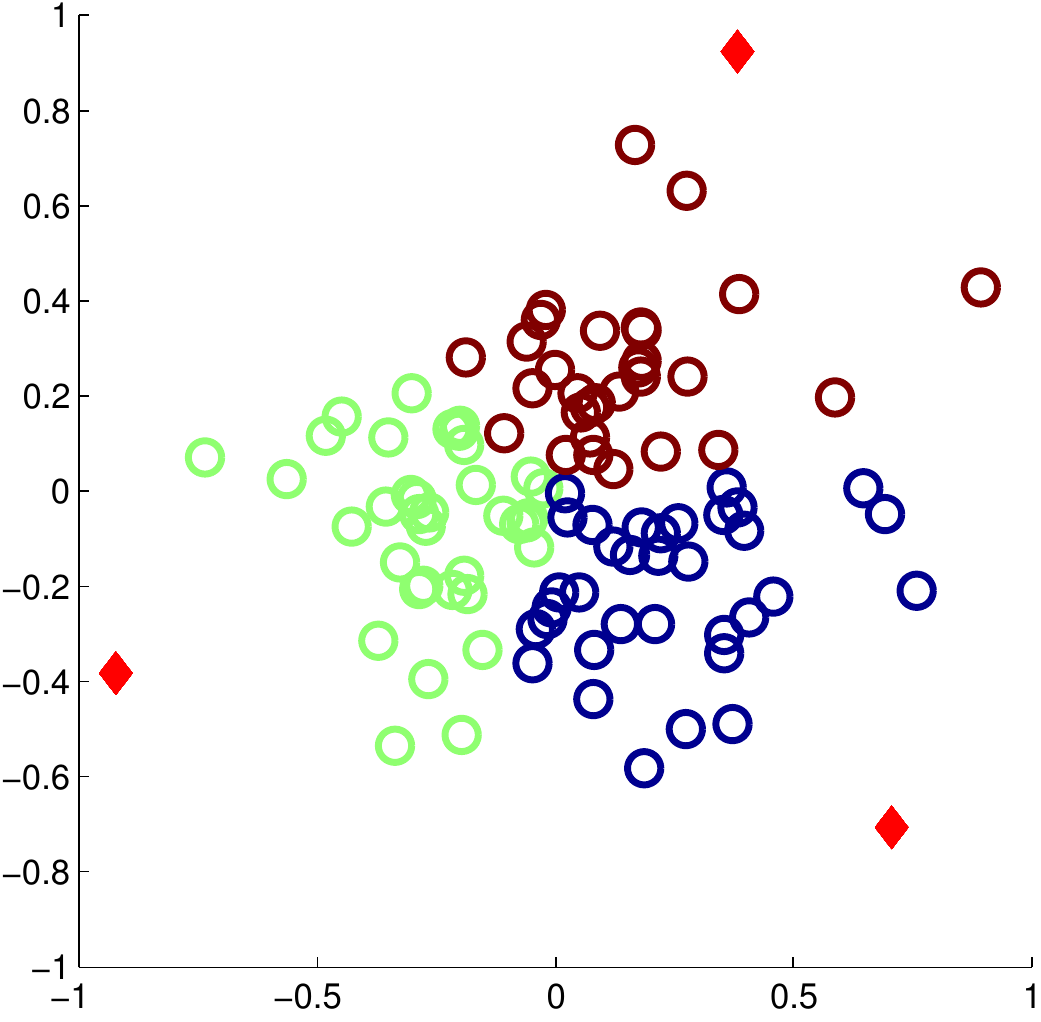}
&
\includegraphics[width=0.3\textwidth]{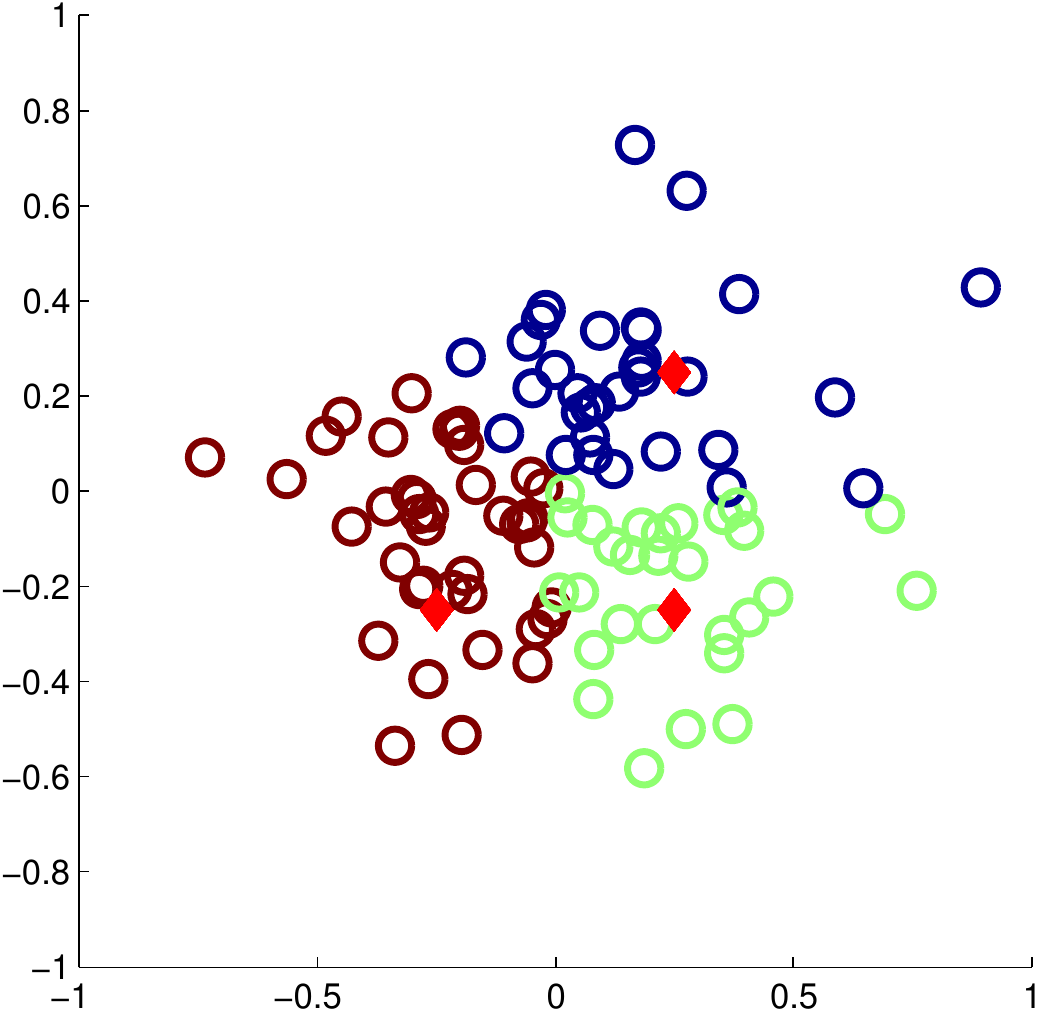}
&
\includegraphics[width=0.3\textwidth]{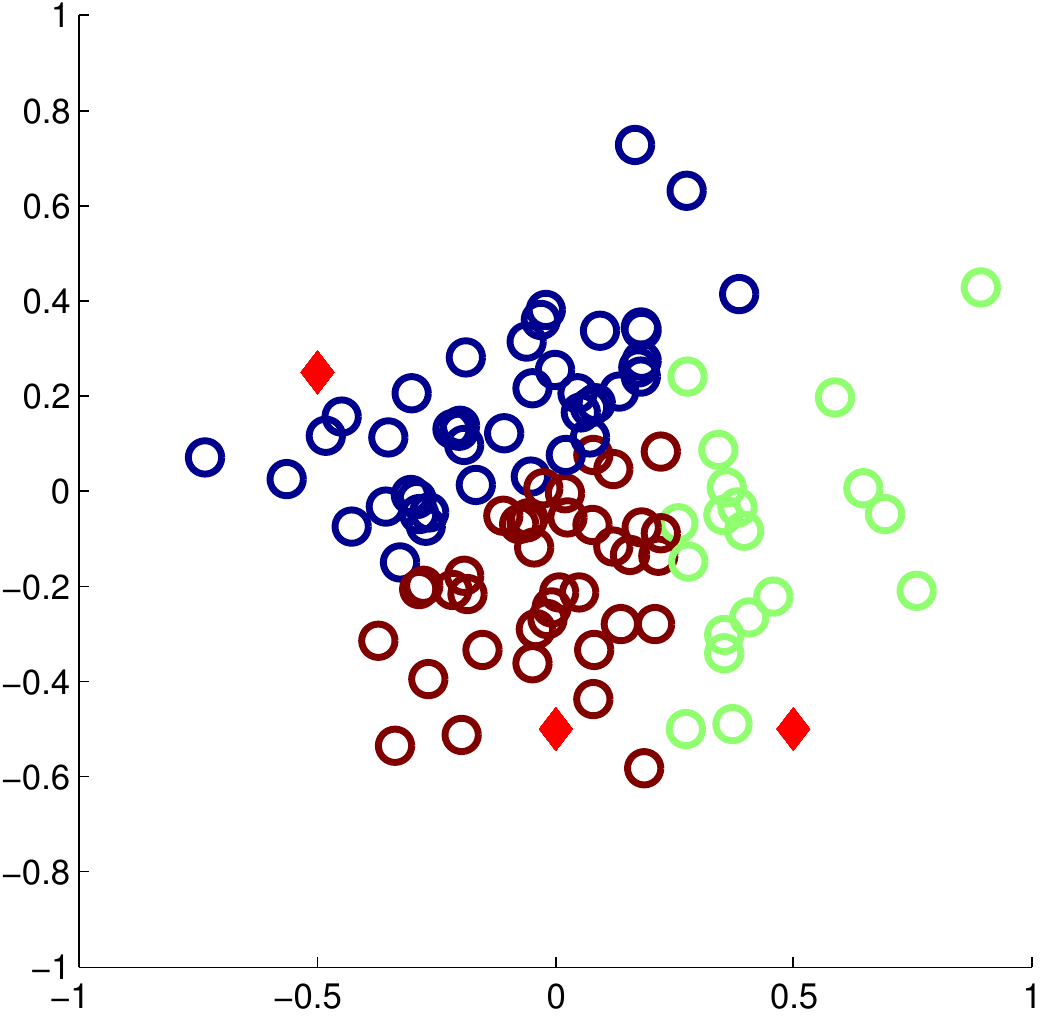}\\

\begin{tikzpicture}[scale=0.5]
\draw [thick] (-0.95 ,1) -- (-1, 1) -- (-1,-1) -- (1,-1) -- (1,-0.95);
\node [below] at ( -1 , -1) {-1};
\node [below] at (  1 , -1) { 1};
\node [left]  at ( -1 , -1) {-1};
\node [left]  at ( -1 ,  1) { 1};

\draw [orange, fill=orange, fill opacity=0.3] (    1.0000,         0) circle [radius=0.05];
\draw [orange, fill=orange, fill opacity=0.3] (    0.9239,    0.3827) circle [radius=0.05];
\draw [orange, fill=orange, fill opacity=0.3] (    0.7071,    0.7071) circle [radius=0.05];
\draw [orange, fill=orange, fill opacity=0.3] (    0.3827,    0.9239) circle [radius=0.05];
\draw [orange, fill=orange, fill opacity=0.3] (    0.0000,    1.0000) circle [radius=0.05];
\draw [orange, fill=orange, fill opacity=0.3] (   -0.3827,    0.9239) circle [radius=0.05];
\draw [orange, fill=orange, fill opacity=0.3] (   -0.7071,    0.7071) circle [radius=0.05];
\draw [orange, fill=orange, fill opacity=0.3] (   -0.9239,    0.3827) circle [radius=0.05];
\draw [orange, fill=orange, fill opacity=0.3] (   -1.0000,    0.0000) circle [radius=0.05];
\draw [orange, fill=orange, fill opacity=0.3] (   -0.9239,   -0.3827) circle [radius=0.05];
\draw [orange, fill=orange, fill opacity=0.3] (   -0.7071,   -0.7071) circle [radius=0.05];
\draw [orange, fill=orange, fill opacity=0.3] (   -0.3827,   -0.9239) circle [radius=0.05];
\draw [orange, fill=orange, fill opacity=0.3] (   -0.0000,   -1.0000) circle [radius=0.05];
\draw [orange, fill=orange, fill opacity=0.3] (    0.3827,   -0.9239) circle [radius=0.05];
\draw [orange, fill=orange, fill opacity=0.3] (    0.7071,   -0.7071) circle [radius=0.05];
\draw [orange, fill=orange, fill opacity=0.3] (    0.9239,   -0.3827) circle [radius=0.05];
\end{tikzpicture}

&

\begin{tikzpicture}[scale=0.5]
\draw [thick] (-0.95 ,1) -- (-1, 1) -- (-1,-1) -- (1,-1) -- (1,-0.95);
\node [below] at ( -1 , -1) {-1};
\node [below] at (  1 , -1) { 1};
\node [left]  at ( -1 , -1) {-1};
\node [left]  at ( -1 ,  1) { 1};

\draw [orange, fill=orange, fill opacity=0.3] (   -1.0000,   -1.0000) circle [radius=0.05];
\draw [orange, fill=orange, fill opacity=0.3] (   -1.0000,    1.0000) circle [radius=0.05];
\draw [orange, fill=orange, fill opacity=0.3] (   -0.7500,   -0.7500) circle [radius=0.05];
\draw [orange, fill=orange, fill opacity=0.3] (   -0.7500,    0.7500) circle [radius=0.05];
\draw [orange, fill=orange, fill opacity=0.3] (   -0.5000,   -0.5000) circle [radius=0.05];
\draw [orange, fill=orange, fill opacity=0.3] (   -0.5000,    0.5000) circle [radius=0.05];
\draw [orange, fill=orange, fill opacity=0.3] (   -0.2500,   -0.2500) circle [radius=0.05];
\draw [orange, fill=orange, fill opacity=0.3] (   -0.2500,    0.2500) circle [radius=0.05];
\draw [orange, fill=orange, fill opacity=0.3] (         0,         0) circle [radius=0.05];
\draw [orange, fill=orange, fill opacity=0.3] (    0.2500,    0.2500) circle [radius=0.05];
\draw [orange, fill=orange, fill opacity=0.3] (    0.2500,   -0.2500) circle [radius=0.05];
\draw [orange, fill=orange, fill opacity=0.3] (    0.5000,    0.5000) circle [radius=0.05];
\draw [orange, fill=orange, fill opacity=0.3] (    0.5000,   -0.5000) circle [radius=0.05];
\draw [orange, fill=orange, fill opacity=0.3] (    0.7500,    0.7500) circle [radius=0.05];
\draw [orange, fill=orange, fill opacity=0.3] (    0.7500,   -0.7500) circle [radius=0.05];
\draw [orange, fill=orange, fill opacity=0.3] (    1.0000,    1.0000) circle [radius=0.05];
\draw [orange, fill=orange, fill opacity=0.3] (    1.0000,   -1.0000) circle [radius=0.05];
\end{tikzpicture}
&

\begin{tikzpicture}[scale=0.5]
\draw [thick] (-0.95 ,1) -- (-1, 1) -- (-1,-1) -- (1,-1) -- (1,-0.95);
\node [below] at ( -1 , -1) {-1};
\node [below] at (  1 , -1) { 1};
\node [left]  at ( -1 , -1) {-1};
\node [left]  at ( -1 ,  1) { 1};

\draw [orange, fill=orange, fill opacity=0.3] (   -0.2500,   -0.5000) circle [radius=0.05];
\draw [orange, fill=orange, fill opacity=0.3] (         0,   -0.5000) circle [radius=0.05];
\draw [orange, fill=orange, fill opacity=0.3] (    0.2500,   -0.5000) circle [radius=0.05];
\draw [orange, fill=orange, fill opacity=0.3] (    0.5000,   -0.5000) circle [radius=0.05];
\draw [orange, fill=orange, fill opacity=0.3] (    0.7500,   -0.5000) circle [radius=0.05];
\draw [orange, fill=orange, fill opacity=0.3] (   -0.5000,   -0.5000) circle [radius=0.05];
\draw [orange, fill=orange, fill opacity=0.3] (   -0.5000,   -0.2500) circle [radius=0.05];
\draw [orange, fill=orange, fill opacity=0.3] (   -0.5000,         0) circle [radius=0.05];
\draw [orange, fill=orange, fill opacity=0.3] (   -0.5000,    0.2500) circle [radius=0.05];
\draw [orange, fill=orange, fill opacity=0.3] (   -0.5000,    0.5000) circle [radius=0.05];
\draw [orange, fill=orange, fill opacity=0.3] (   -0.5000,    0.7500) circle [radius=0.05];
\end{tikzpicture}

\end{tabular}

\caption{\textbf{Euclidean Clustering} with $d=2$, $k=3$, $n=100$ restricted to discrete $\mc{P}$.
Top: Circles corresponding to data points, diamonds corresponding to centers and colors corresponding to clusters. Bottom: Different choices of discrete $\mc{P}$.}
\label{fig:facilities}
\end{figure}


\subsection{Hyperplane Clustering}
By choosing $A_i=a_i$ as row vectors in $\R^d$ and setting $b_i=0$, \eqref{eq:ProblemFormulation} becomes the problem of choosing minimal $\la a_i,x_j\ra^2$ terms. We can interpret this as simultaneously choosing $k$ hyperplanes parameterized by their normal vectors $x_j$ and assigning the points $a_i$ to them according to their weighted angle.

We can uniquely parametrize these hyperplanes by choosing an element $x$ of their complement space which satisfies membership in both 
\begin{equation}
S^d_{\|\cdot \|}=\{x\in \R^d\colon  \|x\|=1\}
\end{equation}
in \emph{any} fixed norm $\|\cdot \|$ and the 'upper halfspace' 
\begin{equation}
H^d_+=\{x\in \R^d\colon x_1\geq 0\}.
\end{equation} 
Note that the norm will weight each point $x\in S^d_{\|\cdot \|}\cap H^d$ by $\|x\|$. In particular, even though any polyhedral approximation of $S^d_{\ell_2}\cap H^d_+$ corresponds to a norm and can be used as $\mc{P}$, this will introduce a slight bias. 

The application of this approach is illustrated by Figure \ref{fig:hyperplane_polygon}.


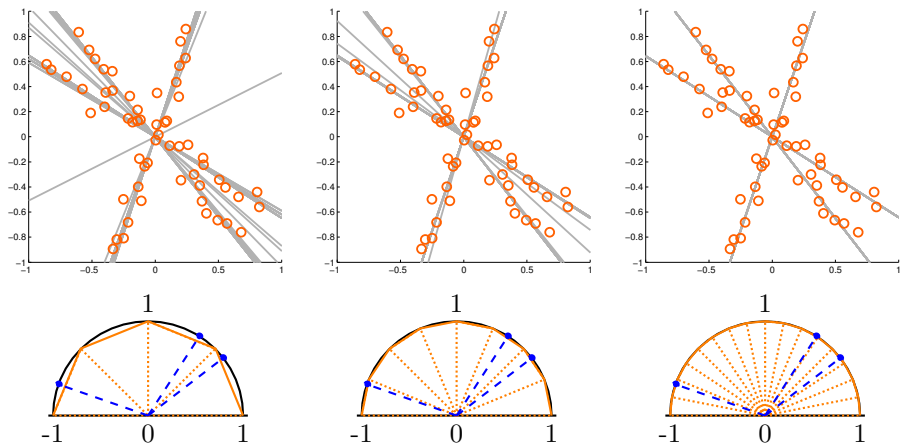
\begin{figure}[!h]

\centering
\begin{tabular}{c c c}

\includegraphics[width=0.3\textwidth]{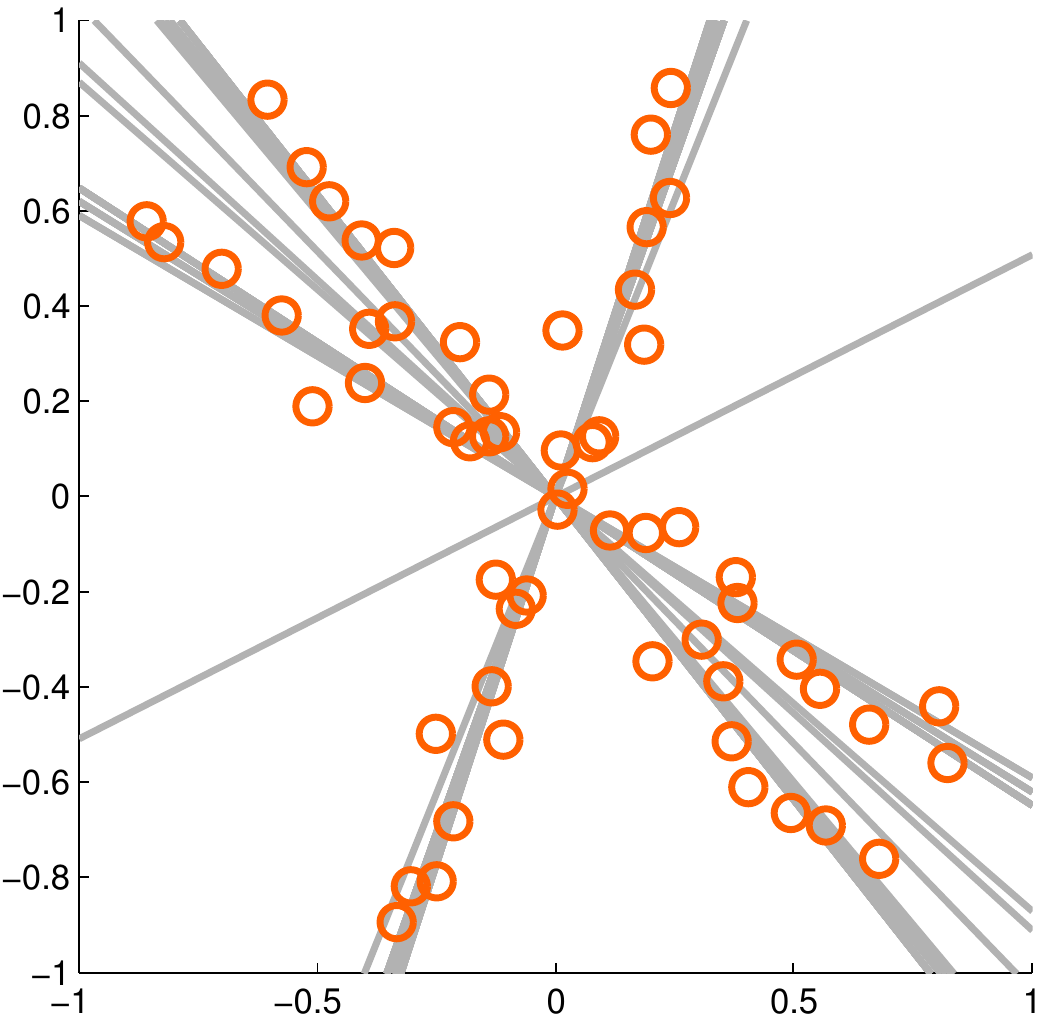}
&
\includegraphics[width=0.3\textwidth]{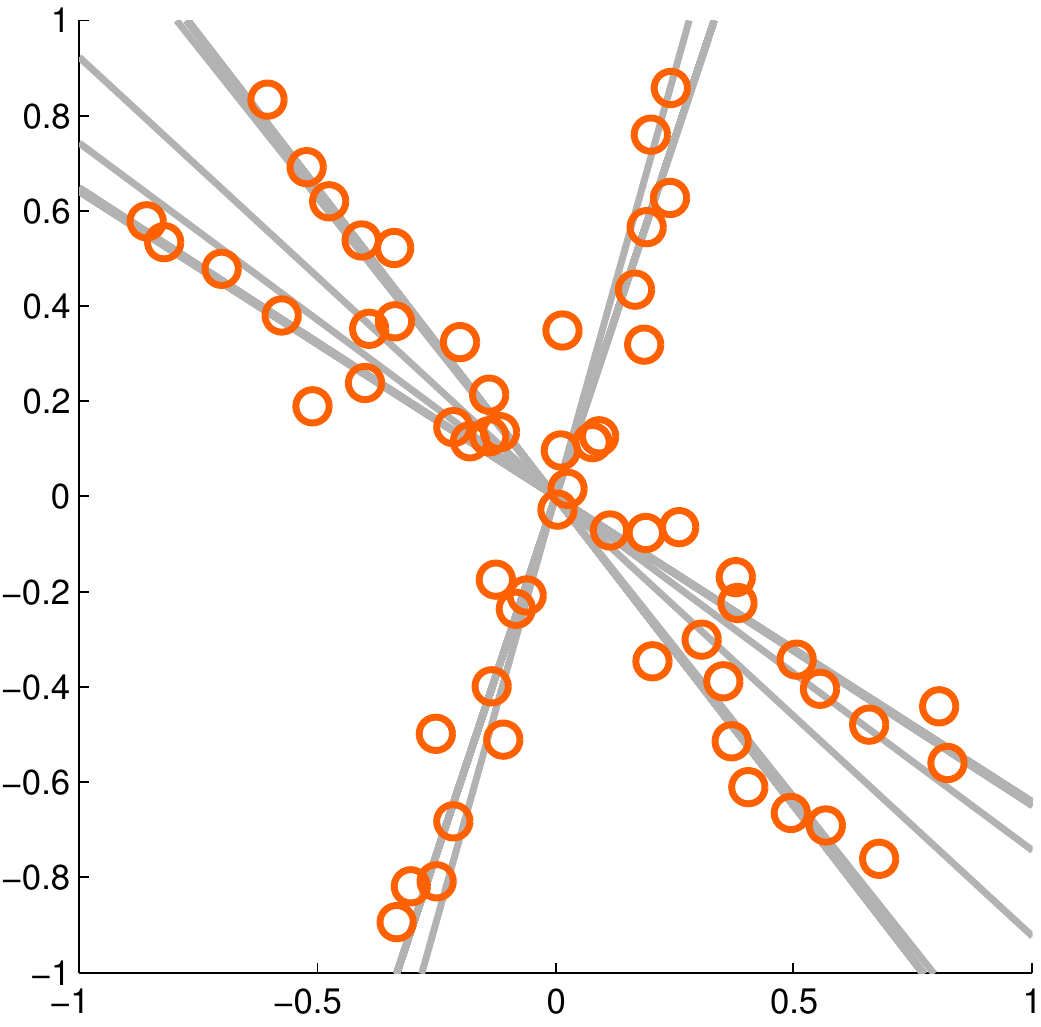}
&
\includegraphics[width=0.3\textwidth]{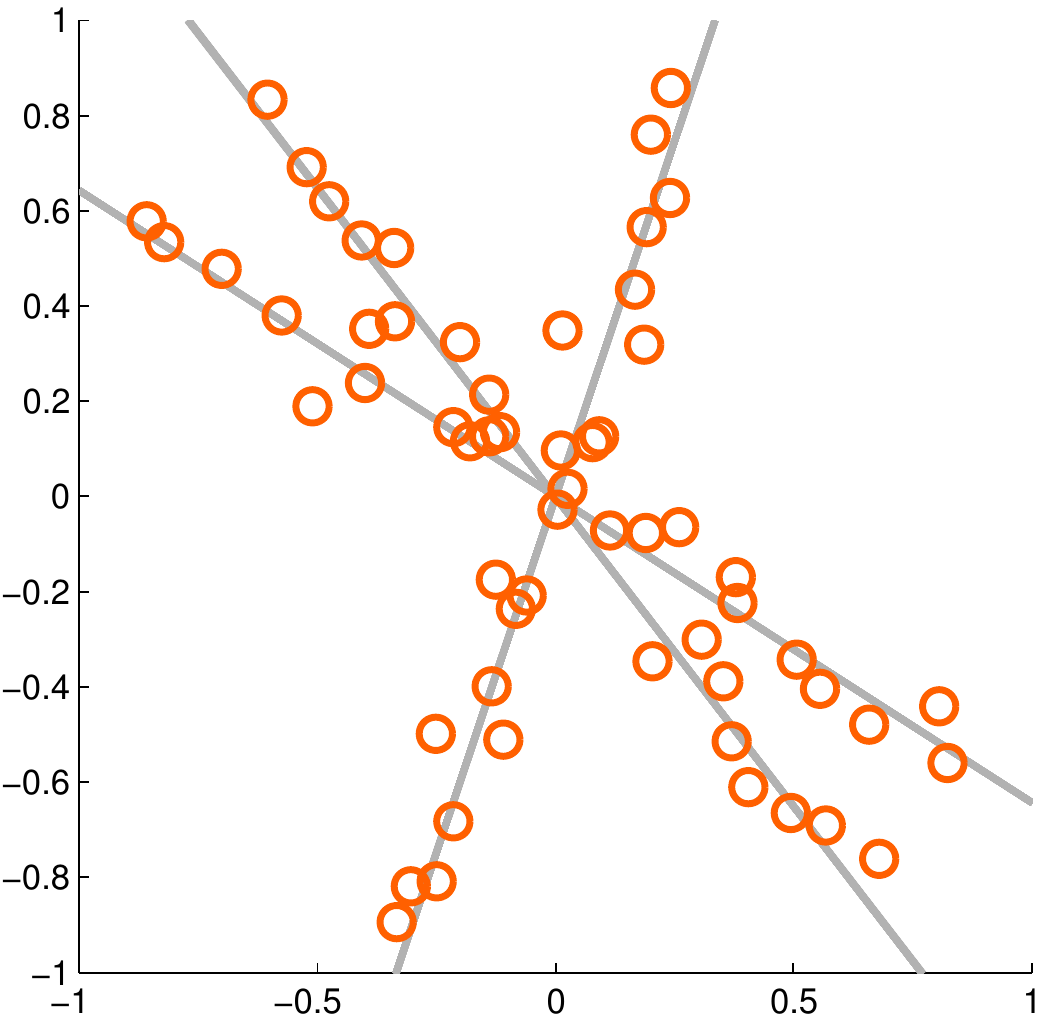}\\

\begin{tikzpicture}[scale=1.25]
\draw [thick,domain=0:180] plot ({cos(\x)}, {sin(\x)});
\draw [thick] (-1.05 ,0) -- (1.05, 0);
\node [below] at ( -1 , 0) {-1};
\node [below] at (  1 , 0) { 1};
\node [above] at (  0,  1) { 1};
\node [below] at (  0,  0) { 0};

\draw [thick, orange] (-1 , 0) -- (-0.71, 0.71) -- (0 , 1) -- (0.71, 0.71) -- (1 , 0);

\draw [orange, thick, densely dotted] (0 , 0) -- (-1, 0) -- (-0.7071, 0.7071) -- (0, 0);
\draw [orange, thick, densely dotted] (0 , 0) -- ( 0, 1) -- ( 0.7071, 0.7071) -- (0, 0);
\draw [orange, thick, densely dotted] (0 , 0) -- ( 1, 0);

\draw [fill,thick,blue,dashed] ( 0, 0) -- ( 0.79, 0.61) circle [radius=0.03];
\draw [fill,thick,blue,dashed] ( 0, 0) -- ( 0.54, 0.84) circle [radius=0.03];
\draw [fill,thick,blue,dashed] ( 0, 0) -- (-0.94, 0.33) circle [radius=0.03];
\end{tikzpicture}
&

\begin{tikzpicture}[scale=1.25]
\draw [thick,domain=0:180] plot ({cos(\x)}, {sin(\x)});
\draw [thick] (-1.05 ,0) -- (1.05, 0);
\node [below] at ( -1 , 0) {-1};
\node [below] at (  1 , 0) { 1};
\node [above] at (  0,  1) { 1};
\node [below] at (  0,  0) { 0};

\draw [thick, orange] 
 (-1, 0) --(-0.9239, 0.3827) -- (-0.7071, 0.7071) -- (-0.3827, 0.9239) --
 ( 0, 1) --( 0.3827, 0.9239) -- ( 0.7071, 0.7071) -- ( 0.9239, 0.3827) -- (1,0);

\draw [orange, thick, densely dotted] (0 , 0) -- (-1, 0) -- (-0.9239, 0.3827) -- (0, 0);
\draw [orange, thick, densely dotted] (0 , 0) -- (-0.7071, 0.7071) -- (-0.3827, 0.9239) -- (0, 0);
\draw [orange, thick, densely dotted] (0 , 0) -- ( 0, 1) -- (0.3827, 0.9239) -- (0, 0);
\draw [orange, thick, densely dotted] (0 , 0) -- ( 0.7071, 0.7071) -- ( 0.9239, 0.3827) -- (0, 0);
\draw [orange, thick, densely dotted] (0 , 0) -- ( 1, 0);

\draw [fill,thick,blue,dashed] ( 0, 0) -- ( 0.79, 0.61) circle [radius=0.03];
\draw [fill,thick,blue,dashed] ( 0, 0) -- ( 0.54, 0.84) circle [radius=0.03];
\draw [fill,thick,blue,dashed] ( 0, 0) -- (-0.94, 0.33) circle [radius=0.03];
\end{tikzpicture}

&

\begin{tikzpicture}[scale=1.25]
\draw [thick,domain=0:180] plot ({cos(\x)}, {sin(\x)});
\draw [thick] (-1.05 ,0) -- (1.05, 0);
\node [below] at ( -1 , 0) {-1};
\node [below] at (  1 , 0) { 1};
\node [above] at (  0,  1) { 1};
\node [below] at (  0,  0) { 0};

\draw [thick, orange] 
 ( 1, 0) --
 ( 0.9808, 0.1951) -- ( 0.9239, 0.3827) -- ( 0.8315, 0.5556) --
 ( 0.7071, 0.7071) --    
 ( 0.5556, 0.8315) -- ( 0.3827, 0.9239) -- ( 0.1951, 0.9808) --
 ( 0, 1) -- 
 (-0.1951, 0.9808) -- (-0.3827, 0.9239) -- (-0.5556, 0.8315) --
 (-0.7071, 0.7071) --     
 (-0.8315, 0.5556) -- (-0.9239, 0.3827) -- (-0.9808, 0.1951) --
 (-1, 0);

\draw [orange, thick, densely dotted] (0 , 0) -- (-1, 0) -- (-0.9808, 0.1951) -- (0, 0);
\draw [orange, thick, densely dotted] (0 , 0) -- (-0.9239, 0.3827) -- (-0.8315, 0.5556) -- (0, 0);
\draw [orange, thick, densely dotted] (0 , 0) -- (-0.7071, 0.7071) -- (-0.5556, 0.8315) -- (0, 0);
\draw [orange, thick, densely dotted] (0 , 0) -- (-0.1951, 0.9808) -- (-0.3827, 0.9239) -- (0, 0);
\draw [orange, thick, densely dotted] (0 , 0) -- ( 0, 1) -- ( 0.1951, 0.9808) -- (0, 0);
\draw [orange, thick, densely dotted] (0 , 0) -- ( 0.5556, 0.8315) -- ( 0.3827, 0.9239) -- (0, 0);
\draw [orange, thick, densely dotted] (0 , 0) -- ( 0.8315, 0.5556) -- ( 0.7071, 0.7071) -- (0, 0);
\draw [orange, thick, densely dotted] (0 , 0) -- ( 0.9808, 0.1951) -- ( 0.9239, 0.3827) -- (0, 0);
\draw [orange, thick, densely dotted] (0 , 0) -- ( 1, 0);

\draw [fill,thick,blue,dashed] ( 0, 0) -- ( 0.79, 0.61) circle [radius=0.03];
\draw [fill,thick,blue,dashed] ( 0, 0) -- ( 0.54, 0.84) circle [radius=0.03];
\draw [fill,thick,blue,dashed] ( 0, 0) -- (-0.94, 0.33) circle [radius=0.03];
\end{tikzpicture}

\end{tabular}

\caption{\textbf{Hyperplane Clustering} with  $d=2$, $k=3$, $n=60$. Top: Circles corresponding to data points and lines corresponding to local estimates for centers parametrized by $\lambda_i$ extracted from $(R3[1])$. 
Bottom: Approximations of $S^2_{\ell_2}\cap H^2_+$ by polygonal lines. For better visibility the ends of each line segment are connected to the origin with an dotted line. Dashed lines end in ground truth angles. }
\label{fig:hyperplane_polygon}

\end{figure}


As an application of Section \ref{sec:semialgebraic_K}, we can also work directly with $S^2_{\ell_2}\cap H^2_+$ by adding the quadratic constraint $x_j^\T x_j=1$ and choosing $\mc{P}\subseteq H^2_+$. Since $S^2_{\ell_2}$ is not polyhedral, we need to use 3-dimensional simplices for $\mc{P}$ in Figure~\ref{fig:hyperplane_smooth}  whereas 2-dimensional simplices sufficed for the polyhedral approximation shown by Figure~\ref{fig:hyperplane_polygon}.


\begin{figure}
\centering
\begin{tabular}{c c c}

\includegraphics[width=0.3\textwidth]{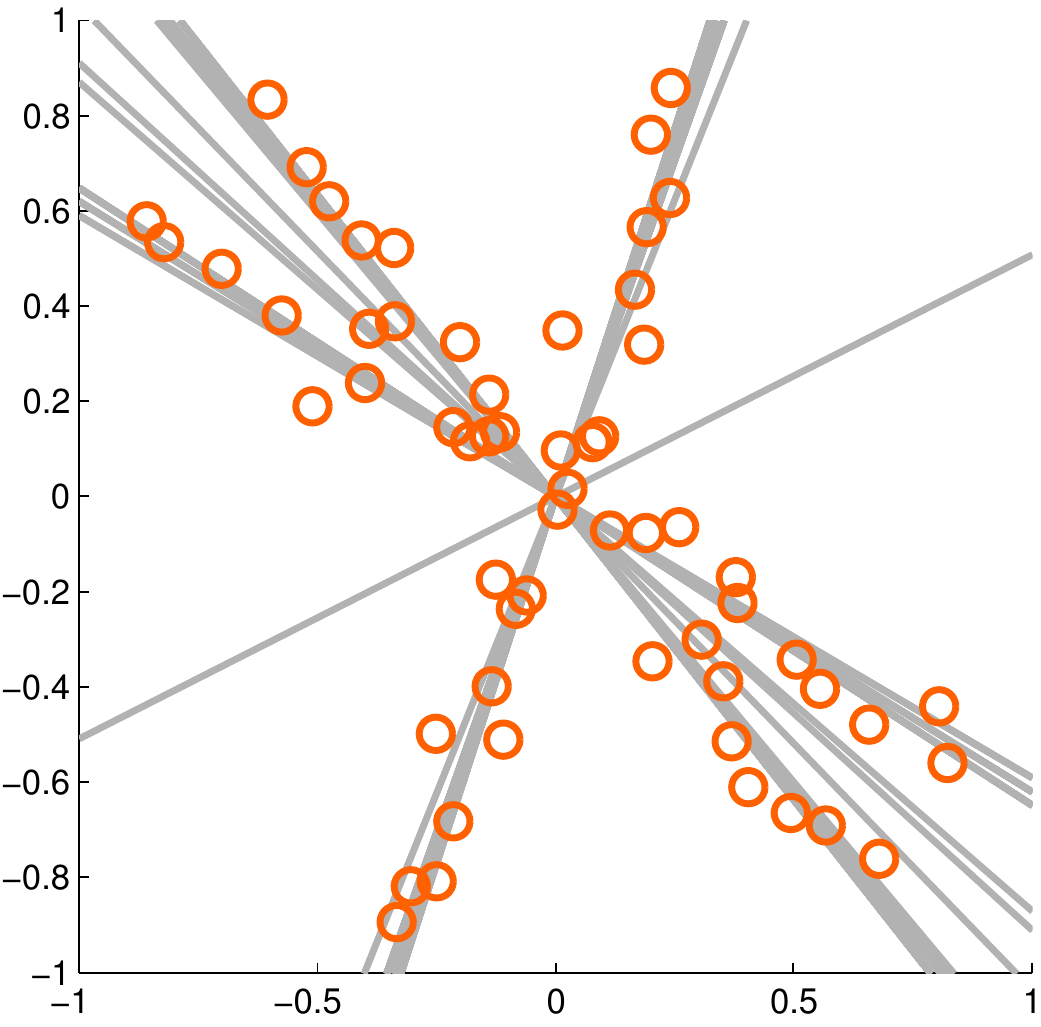}
&
\includegraphics[width=0.3\textwidth]{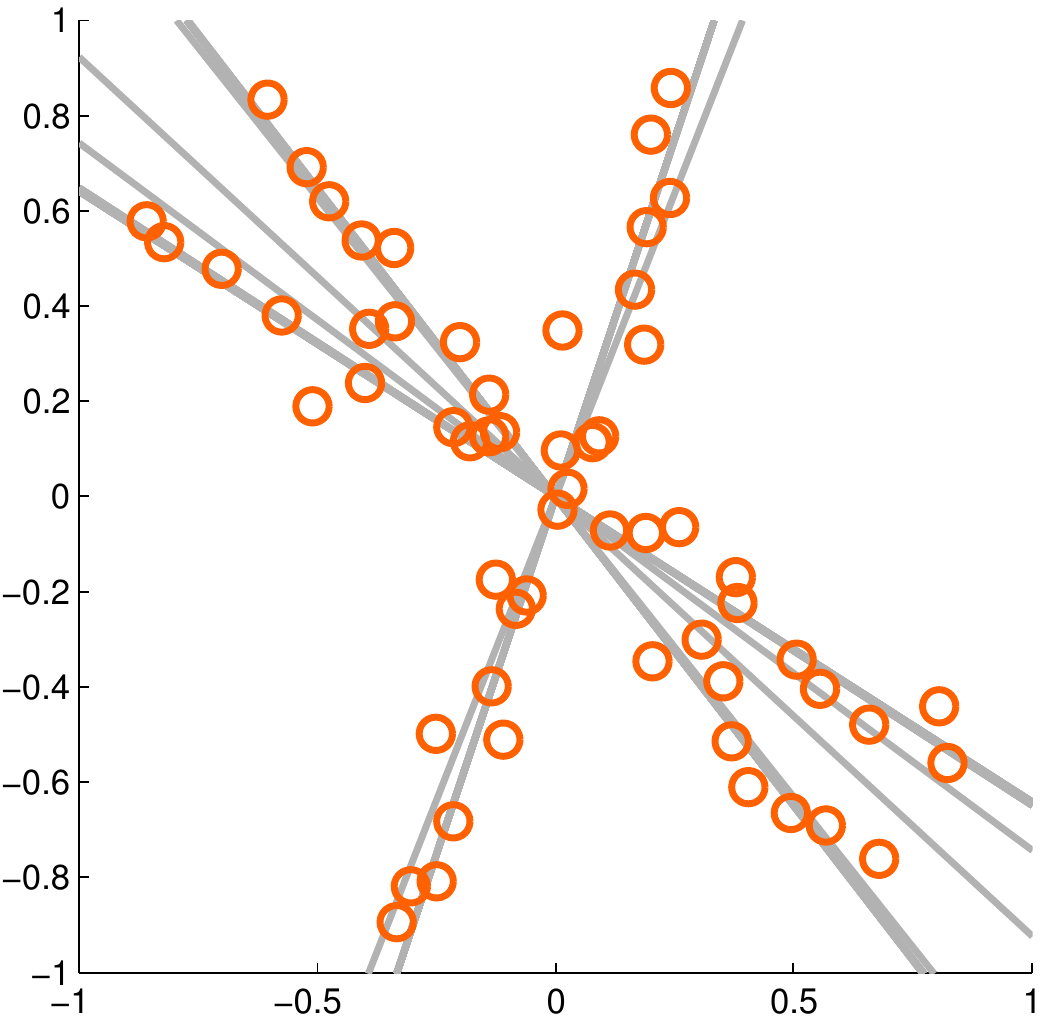}
&
\includegraphics[width=0.3\textwidth]{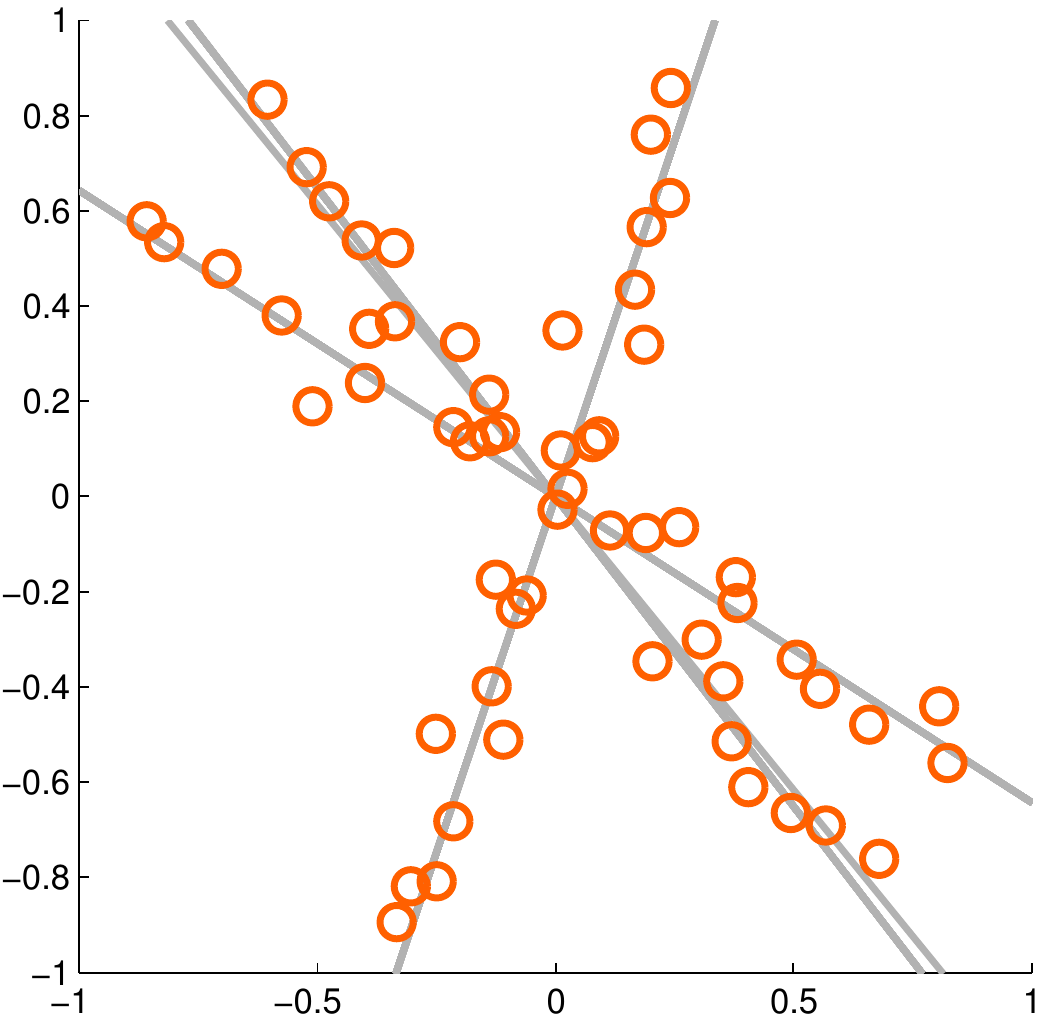}\\

\begin{tikzpicture}[scale=1.25]
\draw [thick] (-1.05 ,0) -- (1.05, 0);

\node [below] at ( -1 , 0) {-1};
\node [below] at (  1 , 0) { 1};
\node [below] at (  0,  0) { 0};

\draw [orange] (-1.0000, 0     ) -- (-0.7071, 0.7071) -- (-1.0243, 0.4243) -- (-1.0000, 0     );
\draw [orange] (-0.7071, 0.7071) -- ( 0     , 1.0000) -- (-0.4243, 1.0243) -- (-0.7071, 0.7071);
\draw [orange] (      0, 1.0000) -- ( 0.7071, 0.7071) -- ( 0.4243, 1.0243) -- (      0, 1.0000);
\draw [orange] ( 0.7071, 0.7071) -- ( 1.0000, 0     ) -- ( 1.0243, 0.4243) -- ( 0.7071, 0.7071);

\draw [orange,domain=0:180] plot ({cos(\x)}, {sin(\x)});

\draw [fill,thick,blue,dashed] ( 0, 0) -- ( 0.79, 0.61) circle [radius=0.03];
\draw [fill,thick,blue,dashed] ( 0, 0) -- ( 0.54, 0.84) circle [radius=0.03];
\draw [fill,thick,blue,dashed] ( 0, 0) -- (-0.94, 0.33) circle [radius=0.03];
\end{tikzpicture}
&

\begin{tikzpicture}[scale=1.25]
\draw [thick] (-1.05 ,0) -- (1.05, 0);

\node [below] at ( -1 , 0) {-1};
\node [below] at (  1 , 0) { 1};
\node [below] at (  0,  0) { 0};

\draw[orange]
   ( 1.0000,         0) -- (    0.9239,    0.3827) -- (    1.1543,    0.2296) -- (    1.0000,         0);
\draw[orange]
   ( 0.9239,    0.3827) -- (    0.7071,    0.7071) -- (    0.9786,    0.6539) -- (    0.9239,    0.3827);
\draw[orange]
   ( 0.7071,    0.7071) -- (    0.3827,    0.9239) -- (    0.6539,    0.9786) -- (    0.7071,    0.7071);
\draw[orange]
   ( 0.3827,    0.9239) -- (    0.0000,    1.0000) -- (    0.2296,    1.1543) -- (    0.3827,    0.9239);
\draw[orange]
   ( 0.0000,    1.0000) -- (   -0.3827,    0.9239) -- (   -0.2296,    1.1543) -- (    0.0000,    1.0000);
\draw[orange]
   (-0.3827,    0.9239) -- (   -0.7071,    0.7071) -- (   -0.6539,    0.9786) -- (   -0.3827,    0.9239);
\draw[orange]
   (-0.7071,    0.7071) -- (   -0.9239,    0.3827) -- (   -0.9786,    0.6539) -- (   -0.7071,    0.7071);
\draw[orange]
   (-0.9239,    0.3827) -- (   -1.0000,    0.0000) -- (   -1.1543,    0.2296) -- (   -0.9239,    0.3827);

\draw [orange,domain=0:180] plot ({cos(\x)}, {sin(\x)});

\draw [fill,thick,blue,dashed] ( 0, 0) -- ( 0.79, 0.61) circle [radius=0.03];
\draw [fill,thick,blue,dashed] ( 0, 0) -- ( 0.54, 0.84) circle [radius=0.03];
\draw [fill,thick,blue,dashed] ( 0, 0) -- (-0.94, 0.33) circle [radius=0.03];
\end{tikzpicture}

&

\begin{tikzpicture}[scale=1.25]
\draw [thick] (-1.05 ,0) -- (1.05, 0);

\node [below] at ( -1 , 0) {-1};
\node [below] at (  1 , 0) { 1};
\node [below] at (  0,  0) { 0};


\draw [orange] 
(    1.0000,         0) -- ( 0.9808,    0.1951) -- ( 1.1885,    0.1171) -- ( 1.0000,         0);
\draw [orange] 
(    0.9808,    0.1951) -- ( 0.9239,    0.3827) -- ( 1.1428,    0.3467) -- ( 0.9808,    0.1951);
\draw [orange] 
(    0.9239,    0.3827) -- ( 0.8315,    0.5556) -- ( 1.0532,    0.5630) -- ( 0.9239,    0.3827);
\draw [orange] 
(    0.8315,    0.5556) -- ( 0.7071,    0.7071) -- ( 0.9231,    0.7576) -- ( 0.8315,    0.5556);
\draw [orange] 
(    0.7071,    0.7071) -- ( 0.5556,    0.8315) -- ( 0.7576,    0.9231) -- ( 0.7071,    0.7071);
\draw [orange] 
(    0.5556,    0.8315) -- ( 0.3827,    0.9239) -- ( 0.5630,    1.0532) -- ( 0.5556,    0.8315);
\draw [orange] 
(    0.3827,    0.9239) -- ( 0.1951,    0.9808) -- ( 0.3467,    1.1428) -- ( 0.3827,    0.9239);
\draw [orange] 
(    0.1951,    0.9808) -- ( 0.0000,    1.0000) -- ( 0.1171,    1.1885) -- ( 0.1951,    0.9808);
\draw [orange] 
(    0.0000,    1.0000) -- (-0.1951,    0.9808) -- (-0.1171,    1.1885) -- ( 0.0000,    1.0000);
\draw [orange] 
(   -0.1951,    0.9808) -- (-0.3827,    0.9239) -- (-0.3467,    1.1428) -- (-0.1951,    0.9808);
\draw [orange] 
(   -0.3827,    0.9239) -- (-0.5556,    0.8315) -- (-0.5630,    1.0532) -- (-0.3827,    0.9239);
\draw [orange] 
(   -0.5556,    0.8315) -- (-0.7071,    0.7071) -- (-0.7576,    0.9231) -- (-0.5556,    0.8315);
\draw [orange] 
(   -0.7071,    0.7071) -- (-0.8315,    0.5556) -- (-0.9231,    0.7576) -- (-0.7071,    0.7071);
\draw [orange] 
(   -0.8315,    0.5556) -- (-0.9239,    0.3827) -- (-1.0532,    0.5630) -- (-0.8315,    0.5556);
\draw [orange] 
(   -0.9808,    0.1951) -- (-0.9239,    0.3827) -- (-1.1428,    0.3467) -- (-0.9808,    0.1951);
\draw [orange] 
(   -0.9808,    0.1951) -- (-1.0000,    0.0000) -- (-1.1885,    0.1171) -- (-0.9808,    0.1951);

\draw [orange, domain=0:180] plot ({cos(\x)}, {sin(\x)});

\draw [fill,thick,blue,dashed] ( 0, 0) -- ( 0.79, 0.61) circle [radius=0.03];
\draw [fill,thick,blue,dashed] ( 0, 0) -- ( 0.54, 0.84) circle [radius=0.03];
\draw [fill,thick,blue,dashed] ( 0, 0) -- (-0.94, 0.33) circle [radius=0.03];
\end{tikzpicture}

\end{tabular}

\caption{\textbf{Hyperplane Clustering} with  $d=2$, $k=3$, $n=60$. Top: Circles corresponding to data points and lines corresponding to local estimates for centers parametrized by $\lambda_i$ extracted from $(R3[1])$. 
Bottom: The semicircle $S^2_{\ell_2}\cap H^2_+$ is covered with triangles in $\mc{P}$. Dashed lines end in ground truth angles. }
\label{fig:hyperplane_smooth}

\end{figure}
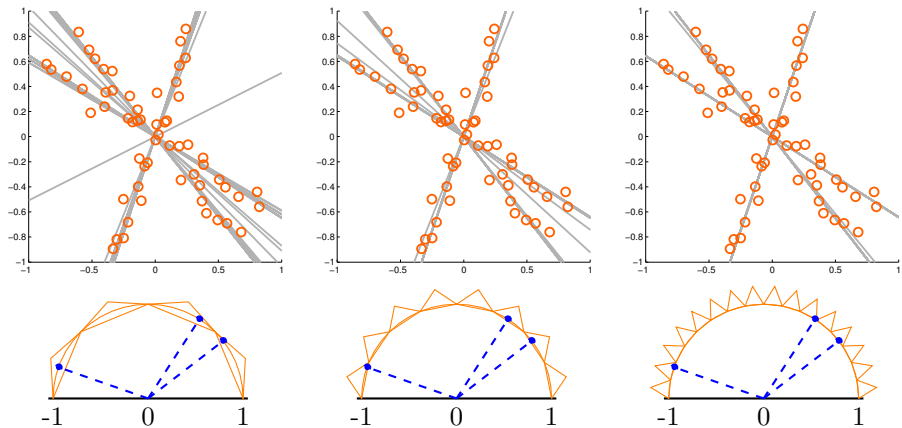


\subsection{Affine Hyperplane Clustering}

We can easily extend the hyperplane clustering to the more general case of clustering affine hyperplanes by changing to homogeneous coordinates. In particular, we can encode the data point $a_i\in \R^d$ as $(a_i, 1)\in \R^{d+1}$ and try to find a hyperplane orthogonal to $(x_j, z_j)\in \R^d$ to get the minimization of terms like
\begin{equation}
\la (a_i, 1), (x_j, z_j)\ra^2 = (\la a_i,x_j\ra +z_j)^2,
\end{equation}
which approximate membership in the affine hyperplane 
\begin{equation}
a_i\in H_{(x_j,z_j)}=\{ a\in \R^d \colon \la a,x_j\ra = -z_j \}.
\end{equation}
Since the manipulation only amounts to lifting the input data $\{a_i\}_{i\in[n]}$, this is just an instance of the Hyperplane Clustering problem in a space with dimension increased by 1. In particular, the problem of Figure \ref{fig:hyperplane_affine} can be computed as an instance of clustering points from $\R^3$ into $2$-dimensional hyperplanes.


\begin{figure}[!h]
\centering
\begin{tabular}{c c c}

\includegraphics[width=0.3\textwidth]{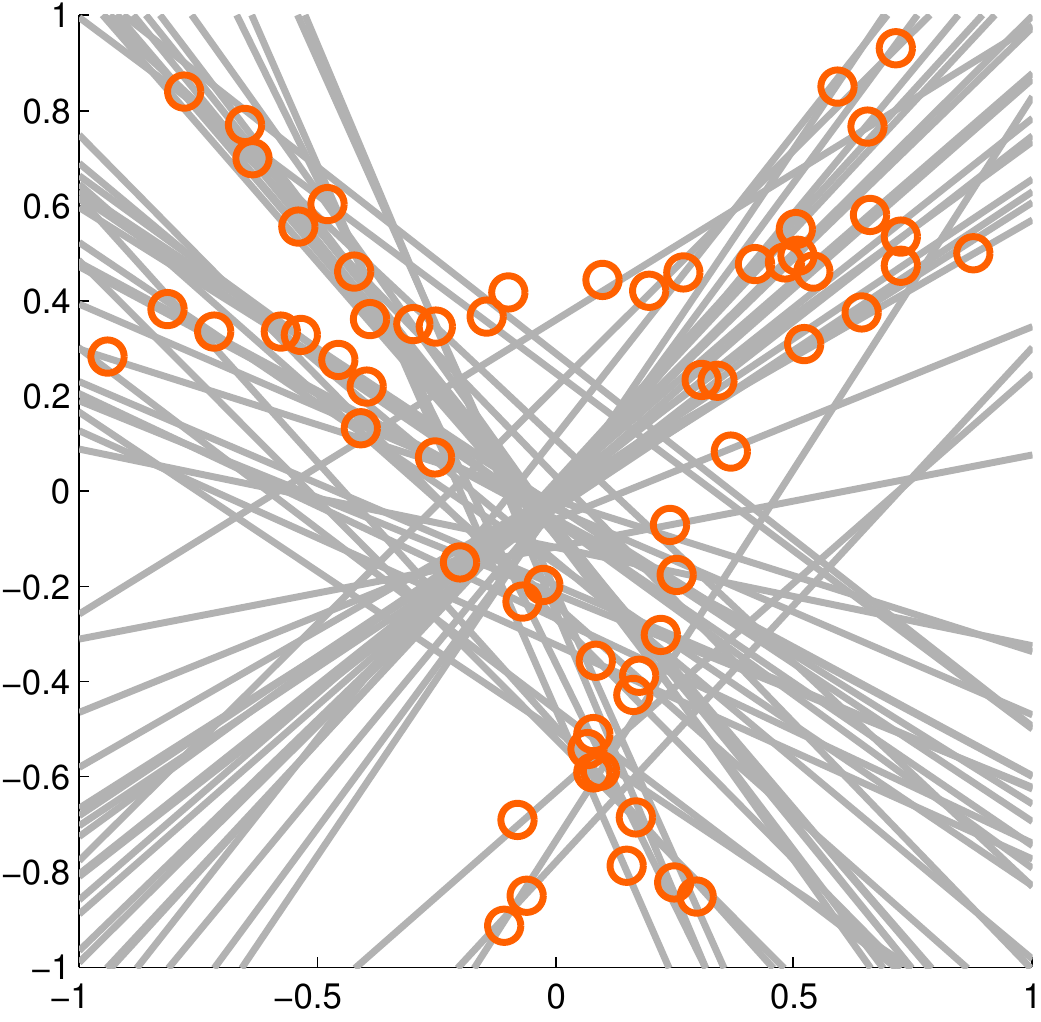}
&
\includegraphics[width=0.3\textwidth]{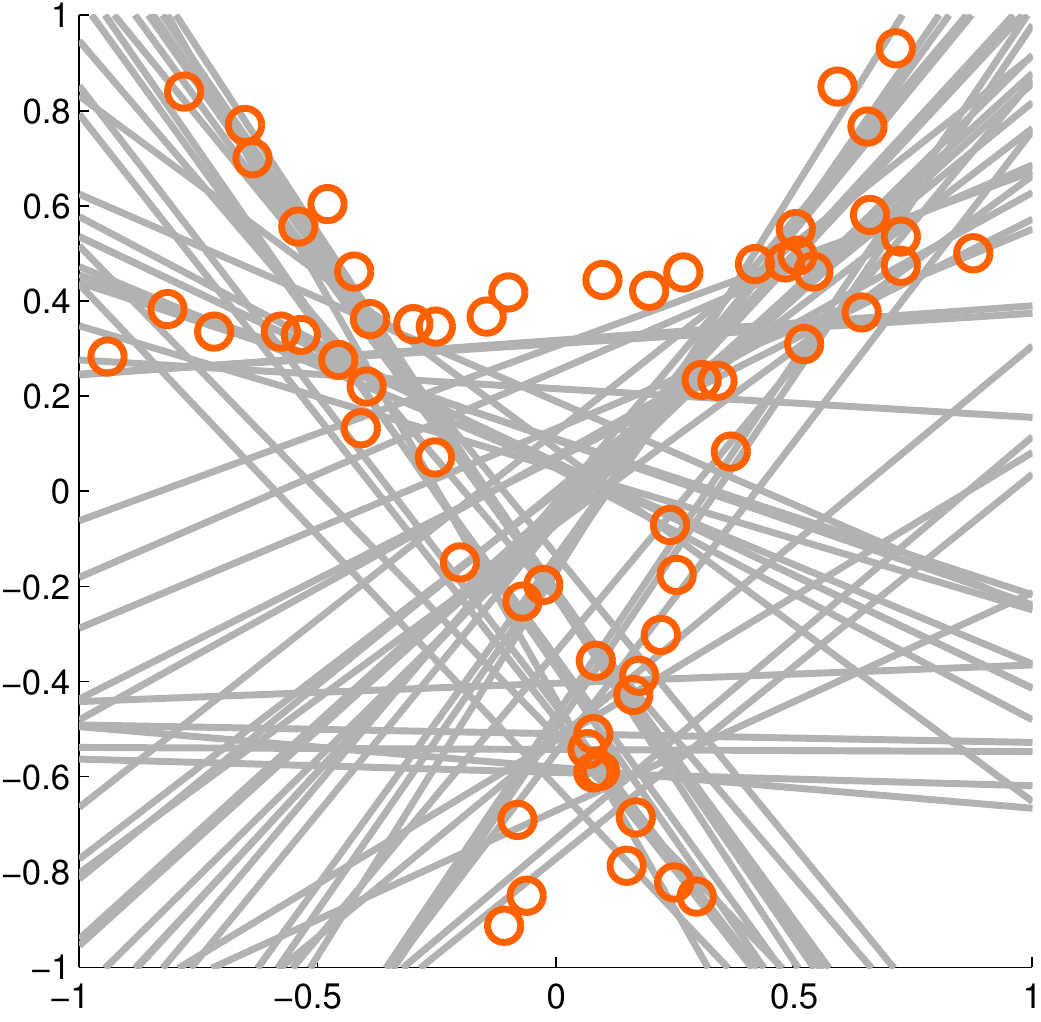}
&
\includegraphics[width=0.3\textwidth]{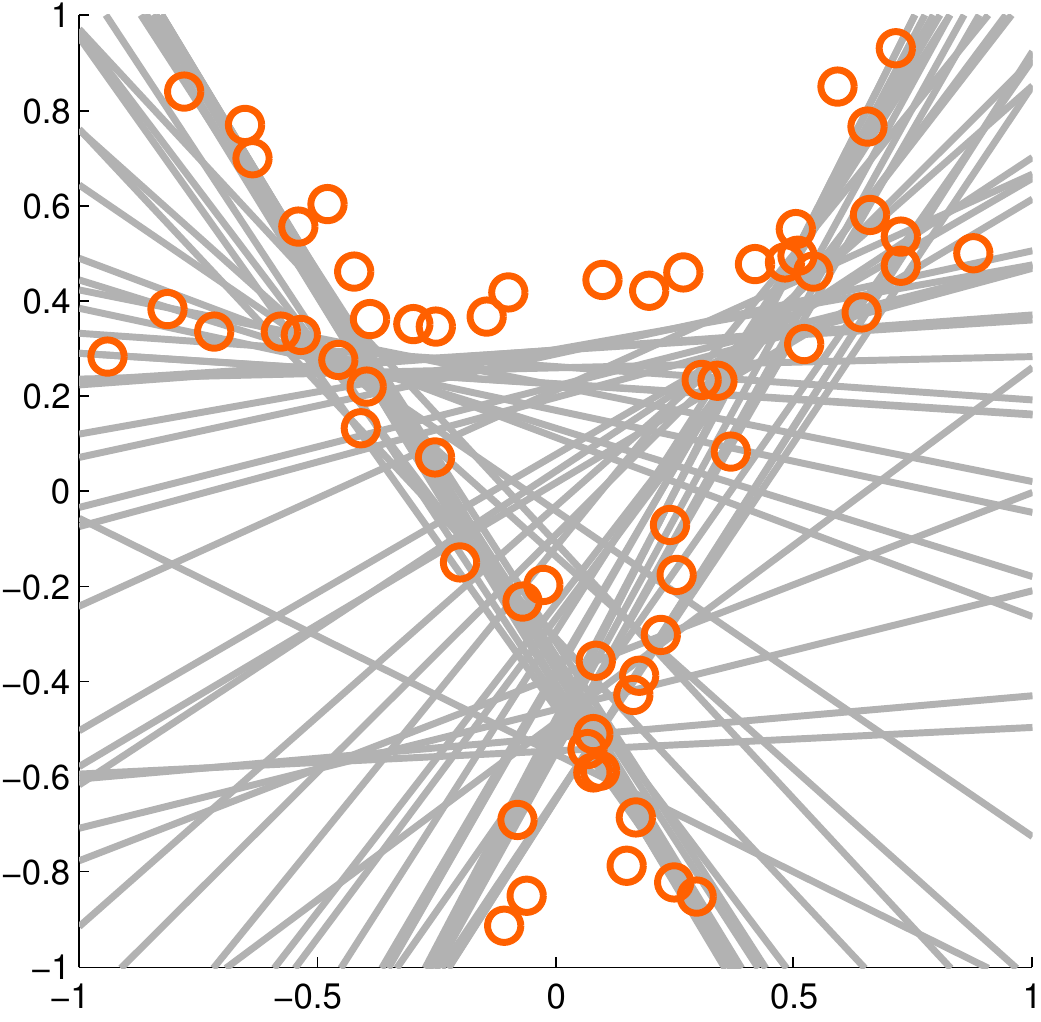}\\

\includegraphics[width=0.3\textwidth]{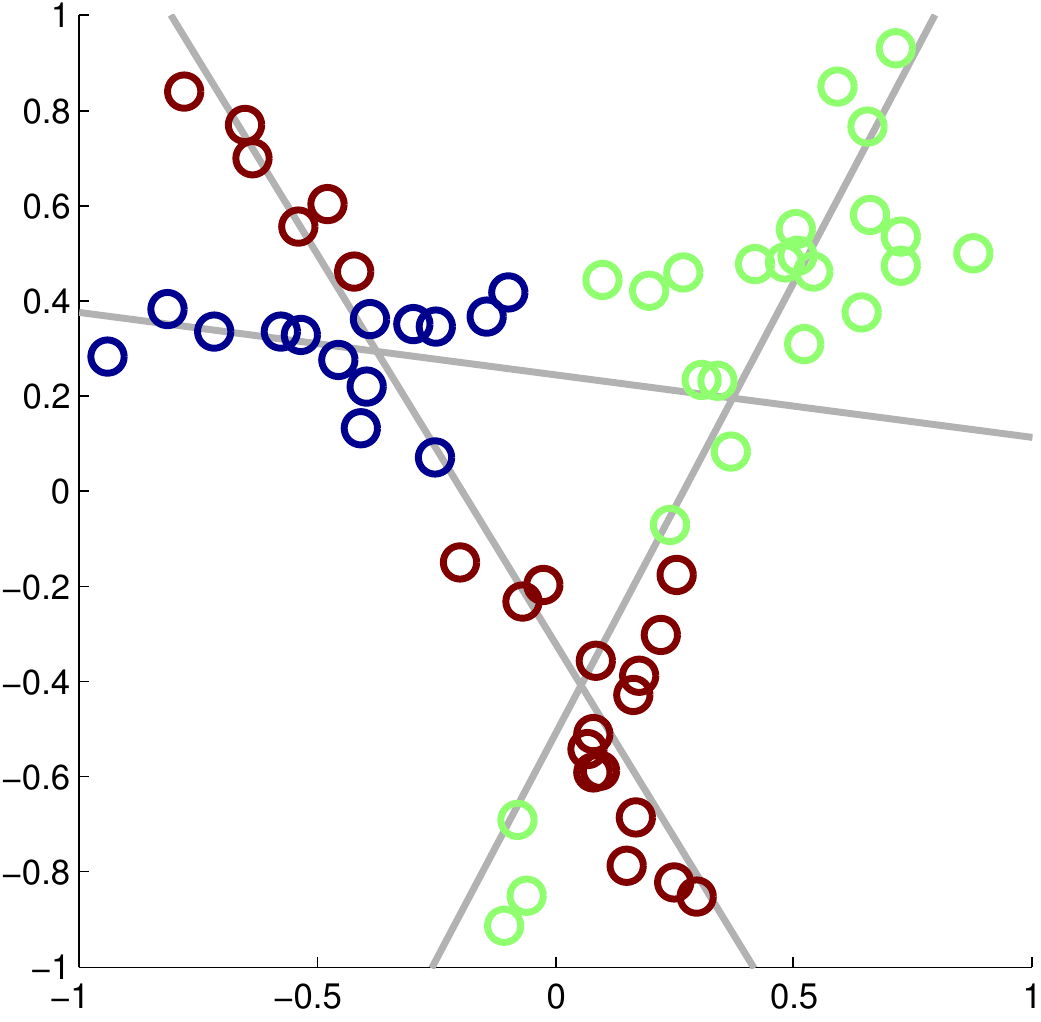}
&
\includegraphics[width=0.3\textwidth]{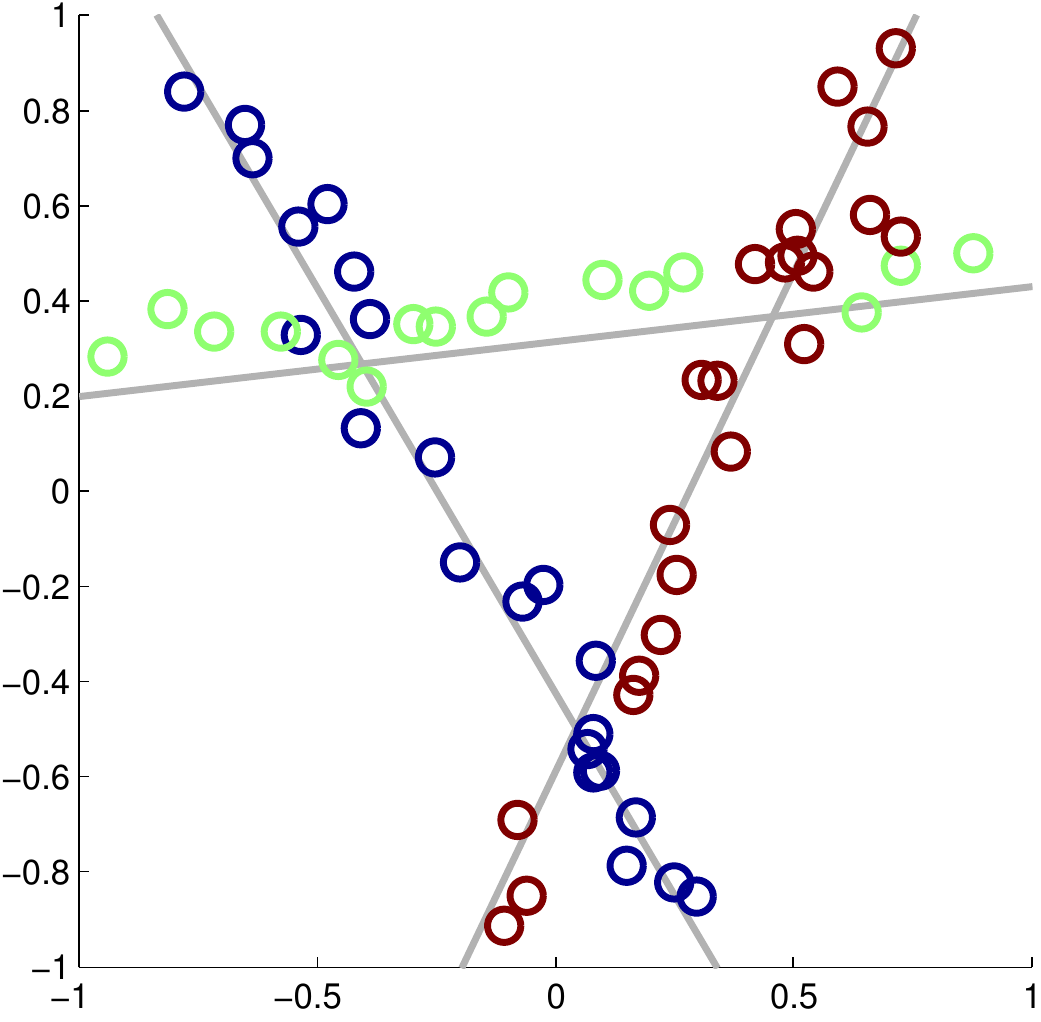}
&
\includegraphics[width=0.3\textwidth]{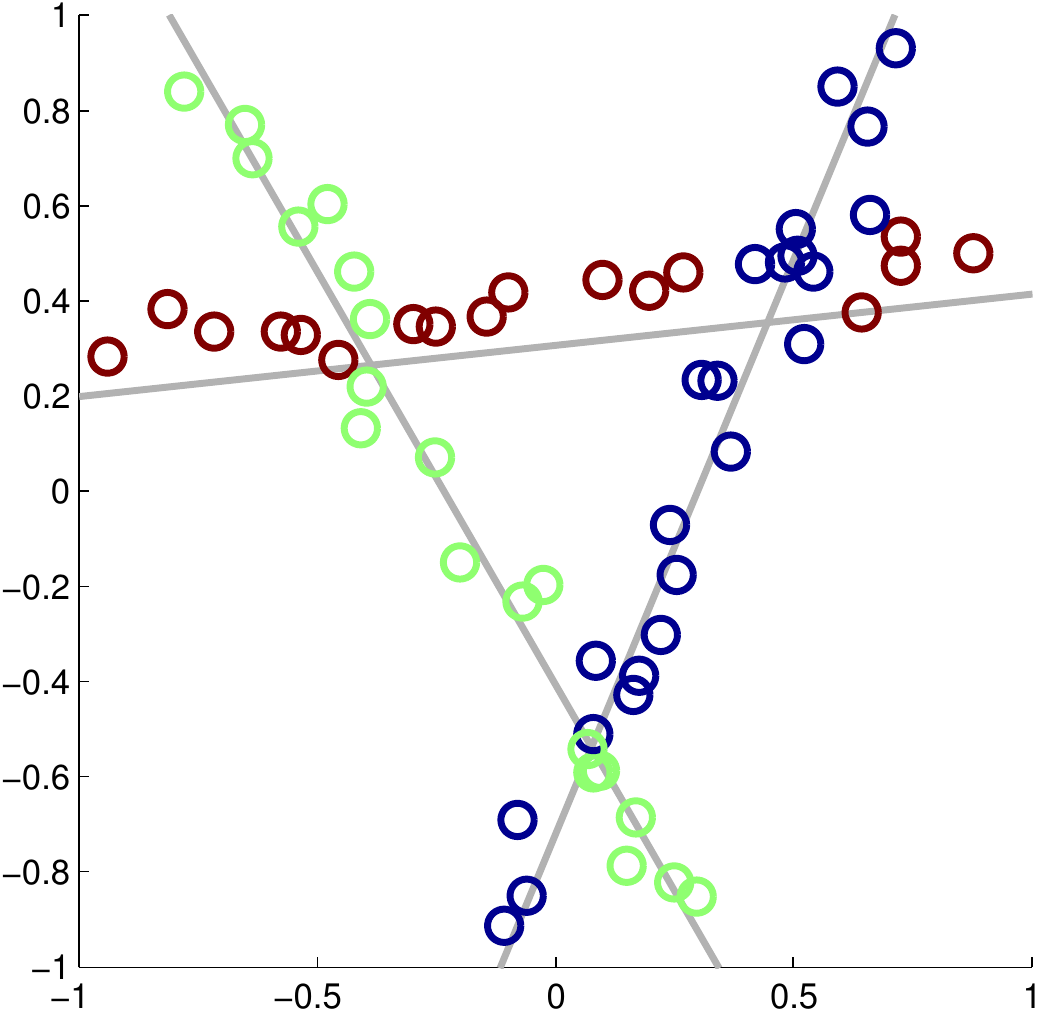}\\
\end{tabular}

\caption{\textbf{Affine Hyperplane Clustering} with  $d=2$, $k=3$, $n=60$ as a special case of \textbf{Hyperplane Clustering} with $d=3$. Top: Circles corresponding to data points and gray lines corresponding to local estimates for centers parametrized by $\lambda_i$ extracted from $(R3[1])$. 
Bottom: Gray lines corresponding to rounded solution of $(R3[1])$ and colored data points according to the extracted clustering. From left to right: Discretization of $S^2_{\ell_2}\cap H^2_+ \times [-0.3, 0.3]$ into $(2\times 8)$ , $(4\times 4)$ and $(8\times 2)$ line segments, where $S^2_{\ell_2}\cap H^2$ is approximated like in figure \ref{fig:hyperplane_polygon}.
}

\label{fig:hyperplane_affine}

\end{figure}

%

\section{Conclusion}

We introduced the concept of separating triangulations for affine subspace clustering problems. Based on this property, a symmetry-free reformulation was deduced for this problem, which allowed us to apply the framework of Lasserres method of moments to construct a hierarchy of convex SDP relaxations. 
We showed how the first step of this hierarchy can be simplified and gave a second hierarchy of relaxation with better computational properties. Based on this, we were able to show experimental results as a proof of concept. 

We hope that this paper  gives some insight into how to remove symmetry from SDPs without reducing them to the invariant space and losing information in this process. While higher steps in the hierarchy may not be tractable for big datasets yet, we hope that this approach may contribute to finding the global solutions for this problem class in the future.

\bibliographystyle{plain}
\bibliography{papersF,papersC}

\end{document}